\documentclass[a4paper,11pt]{amsart}

\usepackage[hmargin=2.245cm,vmargin=2.5cm,includehead]{geometry}
\usepackage[utf8]{inputenc}
\usepackage[T1]{fontenc}
\usepackage{lmodern,microtype}
\usepackage{mathtools,amsthm,amssymb,enumitem,thm-restate,subcaption,hyphenat,graphicx,xcolor}
\usepackage[noadjust]{cite}
\usepackage[linktoc=all,hidelinks,colorlinks,unicode=true,hypertexnames=false]{hyperref}
\usepackage[compress,nameinlink,noabbrev,capitalise]{cleveref}

\hypersetup{
  pdftitle={Burling graphs in graphs with large chromatic number},
  pdfauthor={Tara Abrishami, Marcin Briański, James Davies, Xiying Du, Jana Masaříková, Paweł Rzążewski, Bartosz Walczak},
  linkcolor=black, citecolor=blue!70!black, urlcolor=blue!70!black
}

\makeatletter
\def\@setthanks{\vspace{-\baselineskip}\def\thanks##1{\@par\noindent##1\@addpunct.}\thankses}
\def\@setaddresses{%
  \par\nobreak\begingroup\footnotesize\interlinepenalty\@M\bigskip
  \def\author##1{}%
  \def\address##1##2{%
    \par\nobreak\addvspace\medskipamount\noindent
    \@ifnotempty{##1}{\textit{##1}: }{\ignorespaces##2}%
  }%
  \def\email##1##2{\ifvmode\else\unskip; \fi\textit{email}: \nohyphens{##2}}%
  \addresses\par
  \endgroup
}
\g@addto@macro\@adjustvertspacing{\belowdisplayshortskip\belowdisplayskip}
\makeatother

\newtheorem{theorem}{Theorem}[section]
\newtheorem{lemma}[theorem]{Lemma}
\newtheorem{proposition}[theorem]{Proposition}
\newtheorem{conjecture}[theorem]{Conjecture}

\newlist{enumeratei}{enumerate}{1}
\newlist{enumeratei'}{enumerate}{1}
\newlist{enumeratea}{enumerate}{1}
\setlist[enumeratei]{label=\upshape{(\roman*)}, leftmargin=2.3em}
\setlist[enumeratei']{label=\upshape{(\roman*$'$)}, leftmargin=2.3em}
\setlist[enumeratea]{label=\upshape{(\alph*)}, leftmargin=2.3em}
\setlist[itemize]{label=\upshape{\textbullet}, labelindent=1.5em, leftmargin=2.3em}

\newcommand\NP{\textsf{NP}}
\newcommand\calB{\mathcal{B}}
\newcommand\calC{\mathcal{C}}
\newcommand\calG{\mathcal{G}}
\newcommand\calH{\mathcal{H}}
\newcommand\calM{\mathcal{M}}
\newcommand\calN{\mathcal{N}}
\newcommand\calS{\mathcal{S}}
\newcommand\calU{\mathcal{U}}
\newcommand\calX{\mathcal{X}}
\newcommand\setN{\mathbb{N}}
\newcommand\setR{\mathbb{R}}
\newcommand\set[1]{\{#1\}}
\newcommand\size[1]{{\lvert #1\rvert}}
\newcommand\floor[1]{\lfloor #1\rfloor}
\newcommand\ceil[1]{\lceil #1\rceil}

\DeclareMathOperator\Sub{\mathsf{Sub}}
\DeclareMathOperator\Ext{\mathsf{Ext}}
\DeclareMathOperator\Dup{\mathsf{Dup}}

\let\leq\leqslant
\let\geq\geqslant
\let\setminus\smallsetminus

\title{Burling graphs in graphs with large chromatic number}

\author[T.~Abrishami, M.~Briański, J.~Davies, X.~Du, J.~Masaříková, P.~Rzążewski, B.~Walczak]{Tara Abrishami\and Marcin Briański\and James Davies\and Xiying Du\and Jana Masaříková\and Paweł Rzążewski\and Bartosz Walczak}

\address[Tara Abrishami]{Department of Mathematics, Stanford University, Stanford, USA}
\email{tara.abrishami@stanford.edu}
\address[Marcin Briański]{Department of Theoretical Computer Science, Faculty of Mathematics and Computer Science, Jagiellonian University, Kraków, Poland}
\email{marcin.brianski@doctoral.uj.edu.pl}
\address[James Davies]{Faculty of Mathematics and Computer Science, Leipzig University, Leipzig, Germany}
\email{jgdavies@uwaterloo.ca}
\address[Xiying Du]{School of Mathematics, Georgia Institute of Technology, Atlanta, USA}
\email{xdu90@gatech.edu}
\address[Jana Masaříková]{Institute of Informatics, University of Warsaw, Warsaw, Poland}
\email{jnovotna@mimuw.edu.pl}
\address[Paweł Rzążewski]{Faculty of Mathematics and Information Science, Warsaw University of Technology, Warsaw, Poland, and Institute of Informatics, University of Warsaw, Warsaw, Poland}
\email{pawel.rzazewski@pw.edu.pl}
\address[Bartosz Walczak]{Department of Theoretical Computer Science, Faculty of Mathematics and Computer Science, Jagiellonian University, Kraków, Poland}
\email{bartosz.walczak@uj.edu.pl}

\thanks{A preliminary version of this paper appeared in: Kasper Green Larsen and Barna Saha, editors, \emph{Proceedings of the 2026 Annual ACM-SIAM Symposium on Discrete Algorithms (SODA)}, 3978--3998.\ SIAM, 2026.
The current version corrects an error in the proof of Theorem~1.3 and in the statement and the proof of Theorem~7.3.\\
Tara Abrishami was supported by the National Science Foundation Award Number DMS-2303251 and the Alexander von Humboldt Foundation.
Marcin Briański and Bartosz Walczak were partially supported by the National Science Centre of Poland grant 2019/34/E/ST6/00443.
James Davies was supported by the Alexander von Humboldt Foundation in the framework of the Alexander von Humboldt Professorship of Daniel Kráľ endowed by the Federal Ministry of Education and Research.
Jana Masaříková was supported by the National Science Centre of Poland grant 2022/45/N/ST6/04232.
Paweł Rzążewski was supported by the National Science Centre of Poland grant 2024/54/E/ST6/00094.}

\begin{document}

\begin{abstract}
A graph class is $\chi$-bounded if the only way to force large chromatic number in graphs from the class is by forming a large clique.
In the 1970s, Erdős conjectured that intersection graphs of straight-line segments in the plane are $\chi$-bounded, but this was disproved by Pawlik et~al.\ (2014), who showed another way to force large chromatic number in this class---by triangle-free graphs $B_k$ with $\chi(B_k)=k$ constructed by Burling (1965).
This also disproved the celebrated conjecture of Scott (1997) that classes of graphs excluding induced subdivisions of a fixed graph are $\chi$-bounded.

We prove that in broad classes of graphs excluding induced subdivisions of a fixed graph, including the increasingly more general classes of segment intersection graphs, string graphs, region intersection graphs, and hereditary classes of graphs with finite asymptotic dimension, large chromatic number can be forced \emph{only} by large cliques or large graphs $B_k$.

One corollary is that the hereditary closure of $\{B_k\colon k\geq 1\}$ forms a minimal hereditary graph class with unbounded chromatic number---the second known graph class with this property after the class of complete graphs.
Another corollary is that the decision variant of approximate coloring in the aforementioned graph classes can be solved in polynomial time by exhaustively searching for a sufficiently large clique or copy of $B_k$.
We also discuss how our results along with some results of Chudnovsky, Scott, and Seymour on the existence of colorings can be turned into polynomial-time algorithms for the search variant of approximate coloring in string graphs (with intersection model in the input) and other aforementioned graph classes.
Such an algorithm has not yet been known for any graph class that is not $\chi$-bounded.
\end{abstract}

\maketitle

\section{Introduction}

One of the central research topics in graph theory is aimed at understanding what makes the chromatic number of a graph large.
Every graph $G$ satisfies $\chi(G)\geq\omega(G)$, where $\chi(G)$ and $\omega(G)$ denote, respectively, the chromatic number of $G$ and the maximum size of a clique in $G$.
Deciding whether a given graph $G$ satisfies $\chi(G)\leq k$ is \NP-hard for every $k\geq 3$.
The \emph{$(k,\ell)$-coloring problem}, originating in~\cite{GJ76} and central to the research on promise constraint satisfaction~\cite{KO22}, is to decide whether a given graph $G$ satisfies $\chi(G)\leq k$ or $\chi(G)>\ell$ under the assumption (promise) that one of these cases holds.
It is conjectured to be \NP-hard for all $\ell\geq k\geq 3$; this conjecture has been confirmed when $\ell\leq\max\bigl(2k,\smash{\binom{k}{\floor{k/2}}}\bigr)-1$~\cite{BBKO21,WZ20}.
The $(k,\ell)$-coloring problem was also proved to be hard for all $\ell\geq k\geq 3$ under various perfect-completeness variants of the Unique Games Conjecture \cite{BKLM22,DMR09,GS20}.

If chromatic number greater than $\ell$ could be witnessed by a substructure of bounded size which itself has chromatic number greater than $k$ (such as a clique of size $k+1$), then this would lead to a polynomial-time $(k,\ell)$-coloring algorithm, which would exhaustively look for a witness in the input graph $G$ and report that $\chi(G)>\ell$ if it finds it or $\chi(G)\leq k$ if it does not.
However, for every~$s$, there exist graphs with arbitrarily large chromatic number in which all cycles are longer than~$s$~\cite{Erd59}, which implies that all their induced subgraphs of size at most $s$ are two-colorable.
Still, as the following discussion shows, it makes sense to look for bounded-size witnesses of large chromatic number in restricted classes of graphs.
There are reasons to consider a clique as a ``canonical'' potential witness of large chromatic number:
\begin{itemize}
\item it is the simplest witness possible;
\item the natural linear programming relaxations of $\chi$ (understood as the minimum size of a covering by independent sets) and $\omega$ are dual to each other;
\item in various specific classes of graphs, a large clique is the only way to force large chromatic number.
\end{itemize}

Graphs $G$ for which the equality $\chi=\omega$ holds in $G$ as well as in all induced subgraphs of $G$ are called \emph{perfect}.
Hence, in perfect graphs, chromatic number greater than $k$ is always witnessed by a clique of size $k+1$.
In particular, $k$-coloring can be solved in polynomial time for any constant $k$, and actually the chromatic number can be \emph{computed} in polynomial time via semidefinite programming~\cite{GLS84}.
Perfect graphs have been thoroughly studied; in particular, they have a full structural characterization known as the strong perfect graph theorem~\cite{CRST06}.

A weaker condition (which makes sense only on a class of graphs rather than a single graph) is $\chi$-boundedness: a graph class $\calG$ is \emph{$\chi$-bounded} if there is a function $f\colon\setN\to\setN$ such that for every graph $G\in\calG$, every induced subgraph $H$ of $G$ satisfies $\chi(H)\leq f(\omega(H))$.
Hence, in a $\chi$-bounded graph class, chromatic number greater than $f(k)$ is always witnessed by a clique of size $k+1$, so $(k,f(k))$-coloring can be solved in polynomial time.
Important examples of $\chi$-bounded graph classes include rectangle graphs (intersection graphs of axis-parallel rectangles in the plane)~\cite{AG60}, circle graphs (intersection graphs of chords of a circle)~\cite{Gya85}, and $P_t$-free graphs for any constant $t$~\cite{Gya87}, where we say that a graph $G$ is \emph{$F$-free} if it does not contain an induced subgraph isomorphic to $F$.
Systematic study of $\chi$-boundedness, initiated by Gyárfás~\cite{Gya87}, has been a fruitful area of research in structural graph theory with immense progress in the recent 12 years; see~\cite{SS20c} for a survey.

A major open problem in this area is the Gyárfás--Sumner conjecture~\cite{Gya75,Sum81}, which asserts that the class of $F$-free graphs is $\chi$-bounded for every fixed forest $F$.
(An analogous statement is false when $F$ has a cycle, because of high-chromatic graphs with no short cycles.)
Another problem that used to be open for a long time, due to Gyárfás~\cite{Gya87}, asked whether for every $\ell$, the class of graphs excluding induced cycles of length at least $\ell$ is $\chi$-bounded; this was proved in~\cite{CSS17}.
Scott~\cite{Sco97} proved the following weakening of the Gyárfás--Sumner conjecture: for every forest $F$, the class of graphs excluding all subdivisions of $F$ as induced subgraphs is $\chi$-bounded.
(Actually, it suffices to exclude bounded-length subdivisions for this to hold.)
Motivated by this and the aforementioned problem of Gyárfás on excluded long induced cycles, Scott~\cite{Sco97} conjectured that the class of graphs excluding induced subdivisions of $F$ is $\chi$-bounded for every graph $F$ (not necessarily a forest).

A \emph{string graph} is an intersection graph of arbitrary curves in the plane.
Every intersection graph of arcwise connected geometric objects in the plane can be modeled as a string graph using approximation of the objects by curves.
The class of string graphs attracted significant attention, both from the structural and the algorithmic points of view, as it neatly generalizes various natural graph classes.
For a number of structural results on string graphs, see \cite{Dav25,FP08,FP10,FP12,FP14,FPT11,Mat14,PT02,Tom23}.
A~well-known conjecture, dating back to the 1970s with Erdős in the special case of straight-line segments and due to Kratochvíl and Nešetřil in general, asked whether the class of string graphs is $\chi$-bounded; see \cite[Problem~1.9]{Gya87}, \cite[Problem~2 in Section~9.6]{BMP05}, and~\cite{KN95}.
These conjectures were special cases of Scott's conjecture, because string graphs exclude induced subdivisions of every graph $F$ that is the $1$-subdivision of a non-planar graph (i.e., a non-planar graph with an additional vertex placed ``in the middle'' of every edge)~\cite{Sin66}.

The conjecture on segment/string graphs and Scott's conjecture were disproved by Pawlik et~al.~\cite{PKK+14}, who showed that a particular little-known (at that time) sequence of triangle-free graphs with arbitrarily large chromatic number can be represented as intersection graphs of straight-line segments in the plane.
The said sequence of graphs $(B_k)_{k\geq 1}$ was discovered by Burling~\cite{Bur65} and used in his proof that the class of intersection graphs of axis-parallel boxes in $\setR^3$ is not $\chi$-bounded.
We provide the definition of Burling's sequence later, in \cref{sec:overview}.
For now, think of $(B_k)_{k\geq 1}$ as some sequence of triangle-free graphs such that $\chi(B_k)=k$ for every $k\geq 1$.

\subsection*{Our contribution}

We prove that for string graphs and other special but still broad classes of graphs that are not $\chi$-bounded and exclude induced subdivisions of some fixed graph, Burling's construction is the \emph{only} way to force large chromatic number without creating large cliques.
Formally, for a graph $G$, let $\beta(G)$ denote the maximum $k$ such that $B_k$ is an induced subgraph of $G$.
Note that $\chi(G)\geq\beta(G)$.
We say that a graph class $\calG$ is \emph{Burling-controlled} if there is a function $f\colon\setN\times\setN\to\setN$ such that for every graph $G\in\calG$, every induced subgraph $H$ of $G$ satisfies $\chi(H)\leq f(\omega(H),\beta(H))$.
Equivalently, a class $\calG$ is Burling-controlled if the class of those graphs in $\calG$ that exclude $B_k$ as an induced subgraph is $\chi$-bounded for every $k$.

\begin{theorem}
\label{thm:string-b-bounded}
The class of string graphs is Burling-controlled.
\end{theorem}

\Cref{thm:string-b-bounded} provides a full solution to the aforementioned Erdős and Kratochvíl--Nešetřil problems asking for $\chi$-boundedness of segment and string graphs: even though the classes are not $\chi$-bounded, the reasons for high chromatic number are now fully understood.
This also fully classifies the $\chi$-bounded classes of geometric intersection graphs in the plane and therefore generalizes numerous such $\chi$-boundedness results \cite{AG60,CSS21,Gya85,KK97,KW17,LMPW14,McG96,McG00,RW19a,RW19b,Suk14}.
Apart from string graphs, we also prove that the following graph classes are Burling-controlled:
\begin{itemize}
\item the class of region intersection graphs over every proper minor-closed graph class;
\item every hereditary graph class with finite asymptotic dimension.
\end{itemize}
We proceed with definitions that we need to state these results.

A graph class is \emph{hereditary} if it is closed under removing vertices (i.e., taking induced subgraphs), it is \emph{minor-closed} if it is closed under removing vertices, removing edges, and contracting edges (i.e., taking minors), and it is \emph{proper} if it does not contain all graphs.%
\footnote{We implicitly assume that classes of graphs are closed under graph isomorphism, so that only the structure of the graph matters, not the actual set of vertices.}
A \emph{region intersection model} of a graph $G$ over graph $H$ is a mapping $\mu$ from $V(G)$ to connected sets in $H$ such that $\mu(u)\cap\mu(v)\neq\emptyset$ exactly when $uv$ is an edge of $G$.
A graph $G$ is a \emph{region intersection graph} over a proper minor-closed graph class $\calH$ if $G$ has a region intersection model over some graph $H\in\calH$.%
\footnote{The assumption that $\calH$ is minor-closed can be made without loss of generality, because closing the class under taking minors does not change the class of region intersection graphs derived from it.
The class $\calH$ is assumed to be proper because every graph has a region intersection model over a sufficiently large complete graph.}

Region intersection graphs were introduced by Lee~\cite{Lee17} as a generalization of string graphs, which are precisely the region intersection graphs over planar graphs.
Our next result generalizes \cref{thm:string-b-bounded} to region intersection graphs.

\begin{theorem}
\label{thm:region-b-bounded}
For every proper minor-closed graph class\/ $\calH$, the class of region intersection graphs over\/ $\calH$ is Burling-controlled.
\end{theorem}

The concept of asymptotic dimension was introduced by Gromov~\cite{Gro93} for metric spaces and, in particular, in the context of geometric group theory via groups and their Cayley graphs.
More recently, the same concept has been applied in the context of classes of graphs (with the shortest path distance); this has led to several fascinating results on the interplay between structural graph theory and geometric group theory \cite{BJ25,BBE+23,Dav25,DGHM25,Dis25,DN25,Liu25,Pap23}.
Asymptotic dimension can be considered an analog of chromatic number in coarse graph theory---a recently emerging area of research aimed at studying large-scale geometry of graphs~\cite{GP25}.
From this perspective, a key property of asymptotic dimension is that it is preserved under quasi-isometries, that is, mappings between graphs that approximately preserve distances.

We provide the definition of asymptotic dimension for classes of graphs.
Let $d_G(x,y)$ denote the shortest path distance between vertices $x$ and $y$ in $G$.
The \emph{diameter} of a set of vertices $S$ in $G$ is $\max_{x,y\in S}d_G(x,y)$.
A family $\calU$ of subsets of $V(G)$ is \emph{$\delta$-disjoint} if all distinct sets $X,Y\in\calU$ satisfy $\min_{x\in X}\min_{y\in Y}d_G(x,y)>\delta$, and it is \emph{$(k,\delta)$-disjoint} if $\calU=\bigcup_{i=1}^k\calU_i$ where $\calU_i$ is $\delta$-disjoint for every $i\in[k]$.
A function $f\colon\setN\to\setN$ is a \emph{$d$-dimensional control function} for $G$ if for every $\delta\geq 0$, there is a $(d+1,\delta)$-disjoint partition $\calU$ of $V(G)$ into sets of diameter at most $f(\delta)$ in $G$.
The \emph{asymptotic dimension} of a graph class $\calG$ is the minimum $d$ such that there is a $d$-dimensional control function common for all graphs $G\in\calG$, or it is infinite if no such $d$ exists.

Bonamy et~al.~\cite{BBE+23} proved that every proper minor-closed graph class has asymptotic dimension at most $2$, and every minor-closed graph class excluding some planar graph has asymptotic dimension at most $1$.
Davies~\cite{Dav25} proved that string graphs have asymptotic dimension at most $2$.
Another hereditary graph class recently proved to have finite asymptotic dimension is the class of sphere intersection graphs in $\setR^d$, for every $d\geq 2$~\cite{DGHM25}.
We show a finite bound on the asymptotic dimension of region intersection graphs over any proper minor-closed graph class.

\begin{restatable}{theorem}{thmregionasdim}
\label{thm:region-asdim}
For every class\/ $\calH$ of graphs excluding\/ $K_t$ as a minor, the class of region intersection graphs over\/ $\calH$ has asymptotic dimension at most\/ $2^{t-2}-1$.
\end{restatable}

In view of \cref{thm:region-asdim}, our next result is a further generalization of \cref{thm:region-b-bounded}.

\begin{theorem}
\label{thm:asdim-b-bounded}
Every hereditary graph class with finite asymptotic dimension is Burling-controlled.
\end{theorem}

The most general form of our result involves so-called $\rho$-controlled classes of graphs---a concept originating from Scott~\cite{Sco97}.
For $\rho\geq 2$ and a function $f\colon\setN\to\setN$, a graph $G$ is \emph{$(f,\rho)$-controlled} if for every induced subgraph $H$ of $G$ and every $c\in\setN$, $\max_{v\in V(H)}\chi(N_H^\rho[v])\leq c$ implies $\chi(H)\leq f(c)$, where $N_H^\rho[v]=\set{x\in V(H)\colon d_H(v,x)\leq\rho}$.
A graph class $\calG$ is \emph{$\rho$-controlled} if there is a function $f\colon\setN\to\setN$ such that every graph in $\calG$ is $(f,\rho)$-controlled.
It follows from the definitions that every graph class $\calG$ with finite asymptotic dimension $d$ is $\rho$-controlled for some $\rho$; specifically, if $f'\colon\setN\to\setN$ is a $d$-dimensional control function for $\calG$, then the graphs in $\calG$ are $(c\mapsto(d+1)c,f'(1))$-controlled.

An \emph{$(\leq r)$-subdivision} of a graph $F$ is a subdivision of $F$ in which every edge of $F$ has been subdivided by at most $r$ vertices.
Since every hereditary graph class with finite asymptotic dimension excludes induced $(\leq r)$-subdivisions of some graph for some $r$, the following is a yet further generalization of \cref{thm:asdim-b-bounded}.

\begin{theorem}
\label{thm:control-b-bounded}
For every\/ $\rho\geq 2$ and every graph\/ $F$, every hereditary graph class that is\/ $\rho$-controlled and excludes induced\/ $(\leq 2\rho+5)$-subdivisions of\/ $F$ is Burling-controlled.
\end{theorem}

This leads us to the following resuscitated form of Scott's conjecture.

\begin{conjecture}[Chudnovsky, Scott, and Seymour {\cite[Conjecture~2.1]{CSS21}}]
\label{conj:subdivision-b-bounded}
For every graph\/ $F$, the class of graphs excluding induced subdivisions of\/ $F$ is Burling-controlled.
\end{conjecture}

Chudnovsky, Scott, and Seymour~\cite{CSS21} suggested the above as a statement that one should attempt to refute rather than to prove, to emphasize the fact that Burling's construction is the only one known to disprove Scott's conjecture.
Scott and Seymour~\cite{SS20a} also conjectured that the class of graphs excluding induced subdivisions of $F$ is $2$-controlled for every $F$.
This along with \cref{thm:control-b-bounded} would imply \cref{conj:subdivision-b-bounded}.
Further cases in which we can confirm \cref{conj:subdivision-b-bounded} and open problems in this direction are discussed in \cref{sec:conclusion}.

A moral consequence of our Burling-control results is that the graphs $B_k$ from Burling's sequence should also be considered as ``canonical'' witnesses of large chromatic number, along with cliques.
To our knowledge, these are the first results showing that large chromatic number can be witnessed by bounded-size high-chromatic substructures other than cliques.

An algorithmic consequence of our Burling-control results (just like of all $\chi$-boundedness results) is that in string graphs and the other classes involved, the $(k,\ell)$-coloring problem can be solved in $n^{C(k)}$ time when $\ell\geq f(k,k)$ by exhaustively searching for a $(k+1)$-clique or an induced copy of $B_{k+1}$ and answering ``yes'' if and only if neither can be found.
One may be more interested in producing an actual coloring in polynomial time rather than in solving the decision problem.
Even though computational complexity has never been the focus in the research on $\chi$-boundedness, typical existential $\chi$-boundedness statements can be quite easily transformed into polynomial-time coloring algorithms.
This is also the case for the tools developed by Chudnovsky, Scott, and Seymour~\cite{CSS21} that we use, as well as for our results.
For the sake of simplicity, we do not discuss the running time directly in the statements, but we make sure to structure our proofs in a way that makes transforming them into polynomial-time coloring algorithms pretty straightforward.
We state and discuss such algorithmic versions of \cref{thm:string-b-bounded,thm:region-b-bounded,thm:region-asdim,thm:asdim-b-bounded,thm:control-b-bounded} in \cref{sec:effective}.

Burling's construction is itself used to define the (hereditary) class of \emph{Burling graphs}---induced subgraphs of graphs from the construction.
Burling graphs have various equivalent characterizations~\cite{PT23}, are recognizable in polynomial time~\cite{RW26}, and admit a polynomial-time maximum independent set algorithm~\cite{RW26}.
Since every Burling graph is a string graph, \cref{thm:string-b-bounded} yields the following remarkable property, which was so far known to hold only for the class of complete graphs.

\begin{theorem}
\label{thm:burling-minimal}
The class\/ $\calB$ of Burling graphs is a minimal hereditary class of graphs with unbounded chromatic number (i.e., every hereditary class\/ $\calC\varsubsetneq\calB$ admits a finite bound on the chromatic number).
\end{theorem}

\section{Overview}
\label{sec:overview}

Let $\setN=\set{0,1,2,\ldots}$ and $[n]=\set{1,\ldots,n}$ for all $n\in\setN$.
For a set $X$ and $k\geq 2$, let $\smash{\binom{X}{k}}$ denote the set of all $k$-element subsets of $X$.
Let $N_G(v)$ denote the set of neighbors of a vertex $v$ in a graph $G$, $N_G[v]=N_G(v)\cup\set{v}$, $N_G^2(v)=\set{x\in V(G)\colon d_G(v,x)=2}$, and $N_G^2[v]=N_G[v]\cup N_G^2(v)=\set{x\in V(G)\colon d_G(v,x)\leq 2}$.
The subscript $G$ is omitted if the graph is clear from the context.
A vertex $x$ (a set of vertices $X$) is \emph{complete} to a set of vertices $Y$ in $G$ if $G$ contains all edges between $x$ (every vertex in $X$) and every vertex in $Y$, and it is \emph{anticomplete} to $Y$ if $G$ contains no such edges.
A set of vertices $X$ \emph{covers} a set of vertices $Y$ in $G$ if every vertex in $Y$ has a neighbor in $X$.

\subsection{Burling graphs}

When defining an edge set of a graph, we use the following convenient notation: for two disjoint vertex sets $X$ and $Y$, we let $XY$ denote the set of edges $\set{xy\colon x\in X,\:y\in Y}$.

A \emph{marked graph} is a graph $G$ with a distinguished independent set of vertices denoted by $V_A(G)$.
For a marked graph $G$, we let $V_N(G)=V(G)\setminus V_A(G)$ and $E_N(G)=\set{uv\in E(G)\colon u,v\in V_N(G)}$.
A~marked graph $H$ is an \emph{induced marked subgraph} of $G$, which is denoted by $H\subseteq G$, if $H$ is an induced subgraph of $G$ and $V_A(H)=V_A(G)\cap V(H)$.
A sequence $(G_k)_{k\geq 1}$ of marked graphs is \emph{increasing} if $G_k\subseteq G_{k+1}$ for all $k\geq 1$.
We can always use a marked graph as an ordinary graph, just forgetting about the set $V_A(\cdot)$.

We provide a definition of Burling's sequence $(B_k)_{k\geq 1}$ as a sequence of marked graphs.
We have
\begin{alignat*}{3}
V(B_1)&=\set{a}{,}\qquad&V_A(B_1)&=\set{a}{,}\qquad&E(B_1)&=\emptyset{,}\\
V(B_2)&=\set{b,a}{,}\qquad&V_A(B_2)&=\set{a}{,}\qquad&E(B_2)&=\set{ba}{.}
\end{alignat*}
For $k\geq 2$, once $B_k$ has been defined, $B_{k+1}$ is defined as follows; see \cref{fig:burling-step}.
Let $A=V_A(B_k)$.
For each $a\in A$, let $B^a$ be an isomorphic copy of $B_k$ on a new vertex set, which means that the sets $V(B_k)$ and $V(B^a)$ for all $a\in A$ are pairwise disjoint.
For each $a\in A$ let $A^a=V_A(B^a)$.
For each $a\in A$ and each $v\in A^a$, let $b^a_v$ and $c^a_v$ be two new vertices.
We define
\begin{gather*}
V(B_{k+1})=V_N(B_k)\cup\bigcup_{a\in A}\Bigl(V(B^a)\cup\set{b^a_v,c^a_v\colon v\in A^a}\Bigr){,}\qquad V_A(B_{k+1})=\bigcup_{a\in A}\Bigl(A^a\cup\set{c^a_v\colon v\in A^a}\Bigr){,}\\
E(B_{k+1})=E_N(B_k)\cup\bigcup_{a\in A}\biggl(E(B^a)\cup\bigcup_{v\in A^a}\Bigl(N_{B^a}(v)\set{b^a_v}\cup\set{b^a_vc^a_v}\cup N_{B_k}(a)\set{v,c^a_v}\Bigr)\biggr){.}
\end{gather*}
Identifying every $a\in A$ with some vertex in $A^a$ makes $(B_k)_{k\geq 1}$ an increasing sequence.

\begin{proposition}[Burling \cite{Bur65}]
For every\/ $k\geq 1$, the graph\/ $B_k$ is triangle-free, and\/ $\chi(B_k)=k$.
\end{proposition}

A key insight behind our Burling-control results is that the same class of Burling graphs can be defined using a different increasing sequence $(B'_k)_{k\geq 1}$ of triangle-free marked graphs.
We let $B'_1=B_2$.
For $k\geq 1$, once $B'_k$ has been defined, we define $B'_{k+1}$ as follows; see \cref{fig:new-step}.
Let $A=V_A(B'_k)$.
For each $a\in A$, let $B^a$ be an isomorphic copy of $B'_k$ on a new vertex set, and let $b^a$ be a new vertex, so that the vertex sets $A$, $V_A(B^a)$ for $a\in A$, and $\set{b^a\colon a\in A}$ are pairwise disjoint.
We define
\begin{gather*}
V(B'_{k+1})=V(B'_k)\cup\bigcup_{a\in A}V(B^a)\cup\set{b^a\colon a\in A}{,}\qquad
V_A(B'_{k+1})=A\cup\bigcup_{a\in A}V_A(B^a){,}\\
E(B'_{k+1})=E(B'_k)\cup\bigcup_{a\in A}\Bigl(E(B^a)\cup N_{B'_k}(a)\set{b^a}\cup\set{b^a}V_A(B^a)\Bigr){.}
\end{gather*}

\begin{figure}
  \begin{subfigure}{0.48\textwidth}
    \centering
    \includegraphics[page=1,scale=0.55]{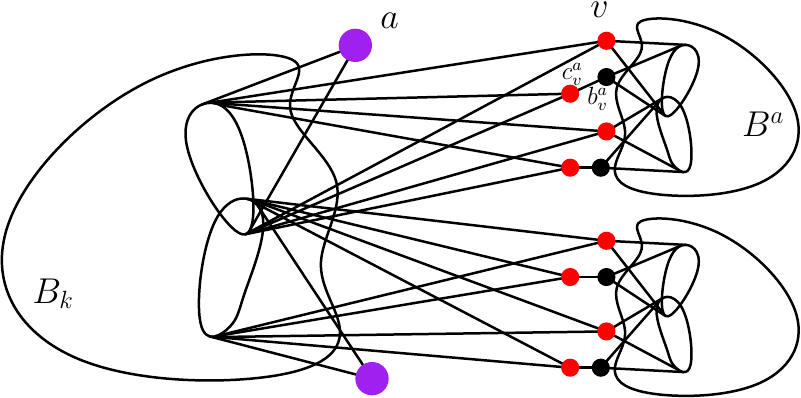}
    \caption{$B_{k+1}$}
    \label{fig:burling-step}
  \end{subfigure}
  \begin{subfigure}{0.48\textwidth}
    \centering
    \includegraphics[page=2,scale=0.55]{./figures.pdf}
    \caption{$B'_{k+1}$}
    \label{fig:new-step}
  \end{subfigure}
  \caption{
    Schematic depiction of the inductive construction of $B_{k+1}$ (\cref{fig:burling-step}) and $B'_{k+1}$ (\cref{fig:new-step}).
    The vertices in $V_A(B_{k+1})$ and $V_A(B'_{k+1})$ are depicted in red.
    Purple vertices are the ones in $V_A(B_k)$ that are not vertices of $B_{k+1}$.
  }
\end{figure}

The following lemma, which we prove in \cref{sec:burling}, essentially asserts that the sequences $(B_k)_{k\geq 1}$ and $(B'_k)_{k\geq 1}$ define the same classes of induced subgraphs.

\begin{restatable}{lemma}{lemburling}
\label{lem:burling}
For every\/ $k\geq 1$, the following holds:
\begin{enumeratei}
\item\label{item:burling-1} $B'_k$ is isomorphic to an induced subgraph of\/ $B_{2k}$;
\item\label{item:burling-2} $B_{k+1}$ is isomorphic to an induced subgraph of\/ $B'_{3\cdot 2^{k-1}-2}$.
\end{enumeratei}
\end{restatable}

\subsection{Excluded subdivisions and control}

A \emph{proper subdivision} of a graph $F$ is a subdivision in which every edge of $F$ has been subdivided at least once.
The following observations are easy.

\begin{proposition}[Sinden {\cite[Section IV]{Sin66}}]
\label{prop:string-subdivision}
String graphs exclude induced proper subdivisions of non-planar graphs.
\end{proposition}

\begin{proposition}[Lee {\cite[Lemma 1.5 of the full version]{Lee17}}]
For every proper minor-closed graph class\/ $\calH$, region intersection graphs over\/ $\calH$ exclude induced proper subdivisions of graphs not in\/~$\calH$.
\end{proposition}

We show an analogous statement for graph classes with finite asymptotic dimension and $(\leq r)$-subdivisions; see \cref{sec:asdim} for the proof.

\begin{restatable}{proposition}{propasdimsubdivision}
\label{prop:asdim-subdivision}
For every hereditary graph class\/ $\calG$ with finite asymptotic dimension and every\/ $r\geq 1$, there is a graph\/ $F$ such that the graphs in\/ $\calG$ exclude induced\/ $(\leq r)$-subdivisions of\/ $F$.
\end{restatable}

Scott~\cite{Sco97}, in his proof that the class of graphs excluding induced subdivisions of $F$ is $\chi$-bounded for every forest $F$, pioneered a two-step approach: first prove that the class is $\rho$-controlled for some (possibly very large) constant $\rho$, and then use this fact to prove $\chi$-boundedness.
This approach was successfully applied in numerous $\chi$-boundedness proofs \cite{CSS16,CSS17,CSS19,CSS21,CSSS20,Dav22,Sco98,SS18,SS19,SS20a,SS20b}.
Since a proper subdivision of $K_{t,t}$ with sufficiently large $t$ contains an induced subdivision of every graph, the following lemma asserts that being $\rho$-controlled for some $\rho\geq 2$ is equivalent to being $2$-controlled for classes of graphs with excluded induced subdivisions of a fixed graph.

\begin{lemma}[Chudnovsky, Scott, Seymour {\cite[1.10]{CSS21}}]
\label{lem:control-reduction}
For all\/ $t\geq 1$, $\rho\geq 2$, and\/ $f\colon\setN\to\setN$, there is\/ $f'\colon\setN\to\setN$ such that every\/ $(f,\rho)$-controlled graph excluding induced proper\/ $(\leq\rho+2)$-subdivisions of\/ $K_{t,t}$ is\/ $(f',2)$-controlled.
\end{lemma}

Chudnovsky, Scott, and Seymour~\cite{CSS21} proved that the class of string graphs is $2$-controlled by first proving that it is $3$-controlled using a simple direct argument and then applying \cref{lem:control-reduction} to infer that it is $2$-controlled.

\begin{theorem}[Chudnovsky, Scott, Seymour {\cite[12.5]{CSS21}}]
\label{thm:string-2-controlled}
The class of string graphs is\/ $2$-controlled.
\end{theorem}

Using \cref{prop:asdim-subdivision} and \cref{lem:control-reduction}, we prove the following; see \cref{sec:asdim}.

\begin{restatable}{theorem}{thmasdimcontrol}
\label{thm:asdim-2-controlled}
Every hereditary graph class with finite asymptotic dimension is\/ $2$-controlled.
\end{restatable}

A \emph{$1$-subdivision} of a graph $F$ is a subdivision in which every edge of $F$ has been subdivided exactly once.
Chudnovsky, Scott, and Seymour \cite[1.9]{CSS21} proved that every $2$-controlled class of graphs excluding induced $1$-subdivisions of $K_{t,t}$ and excluding a special kind of Burling graph called ``tree of lamps'' as an induced subgraph is $\chi$-bounded.
The main contribution of this work is as follows.

\begin{theorem}
\label{thm:b-bounded}
For every\/ $t\geq 1$, every\/ $2$-controlled class of graphs excluding induced\/ $1$-subdivisions of\/ $K_{t,t}$ is Burling-controlled.
\end{theorem}

All other Burling-control results of this paper are corollaries to \cref{thm:b-bounded}.

\begin{proof}[Proof of \cref{thm:string-b-bounded}]
By \cref{thm:string-2-controlled}, the class of string graphs is $2$-controlled.
Since the graph $K_{3,3}$ is non-planar, \cref{prop:string-subdivision} implies that no string graph contains an induced $1$-subdivision of $K_{3,3}$.
The theorem now follows from \cref{thm:b-bounded}.
\end{proof}

\begin{proof}[Proof of \cref{thm:asdim-b-bounded}]
Let $\calG$ be a hereditary graph class with finite asymptotic dimension.
By \cref{thm:asdim-2-controlled}, $\calG$ is $2$-controlled.
By \cref{prop:asdim-subdivision} applied with $r=3$, there is a graph $F$ such that the graphs in $\calG$ exclude induced $(\leq 3)$-subdivisions of $F$.
Let $t=\max(\size{V(F)},\size{E(F)})$.
The theorem follows from \cref{thm:b-bounded}, as a $1$-subdivision of $K_{t,t}$ contains an induced $3$-subdivision of $F$.
\end{proof}

\begin{proof}[Proof of \cref{thm:control-b-bounded}]
Let $\calG$ be a $\rho$-controlled class of graphs excluding induced $(\leq 2\rho+5)$-subdivisions of a fixed graph $F$.
Let $t=\max(\size{V(F)},\size{E(F)})$.
Since a proper $(\leq\rho+2)$-subdivision of $K_{t,t}$ contains an induced $(\leq 2\rho+5)$-subdivision of $F$, \cref{lem:control-reduction} implies that $\calG$ is $2$-controlled.
The theorem now follows from \cref{thm:b-bounded}.
\end{proof}

We prove \cref{thm:region-asdim} in \cref{sec:region}.
\Cref{thm:region-b-bounded} then follows directly from \cref{thm:region-asdim,thm:asdim-b-bounded}.

\subsection{Multicovers}

Multicovers are a useful tool for proving bounds on the chromatic number in $2$-controlled graph classes, developed in the works of Chudnovsky, Scott, and Seymour~\cite{CSS17,CSS21}.

A \emph{multicover} in $G$ with \emph{index set} $I\subseteq\setN$ is a triple $(\set{x_i}_{i\in I},\set{N_i}_{i\in I},C)$ with $x_i\in V(G)$ for $i\in I$, $N_i\subseteq V(G)$ for $i\in I$, and $C\subseteq V(G)$, such that the following conditions hold:
\begin{itemize}
\item the sets $\set{x_i}$ for $i\in I$, $N_i$ for $i\in I$, and $C$ are pairwise disjoint;
\item for every $i\in I$, the vertex $x_i$ is complete to $N_i$ and anticomplete to $\bigcup_{j\in I,j>i}(\set{x_j}\cup N_j)\cup C$;
\item for every $i\in I$, the set $N_i$ covers $C$.
\end{itemize}
The multicover is \emph{independent} if it satisfies the following additional condition: for every $j\in I$, the vertex $x_j$ is anticomplete to $\bigcup_{i\in I,i<j}N_i$.
For notational convenience, we allow the sets $N_i$ to be empty in a multicover, in which case $C$ must also be empty if $I\neq\emptyset$.

Let $G$ be an $(f,2)$-controlled graph with $\omega(G)\leq k$ for some function $f\colon\setN\to\setN$ and some $k\geq 1$.
If we are to prove an upper bound on $\chi(G)$, we can set up induction on $k$ and therefore assume, as the induction hypothesis, that there is a constant $c$ such that $\chi(H)\leq c$ for every induced subgraph $H$ of $G$ with $\omega(H)\leq k-1$.
So, in particular, $\chi(N[v])\leq c+1$ for every $v\in V(G)$, because $\omega(N(v))\leq k-1$.
For every induced subgraph $H$ of $G$, since $G$ is $(f,2)$-controlled and $\chi(N_H^2[x])\leq\chi(N_H^2(x))+c+1$, we have $\chi(H)\leq f(c'+c+1)$ whenever $\max_{x\in V(H)}\chi(N_H^2(x))\leq c'$.
Let $c_0\geq 0$ and $c_\ell=f(c_{\ell-1}+c+1)$ for $\ell\geq 1$.

Now, there is a natural way of iteratively extracting a long multicover in $G$ when $\chi(G)$ is sufficiently large.
Let $G_0=G$, and suppose $\chi(G_0)>c_\ell$.
By the above, there is $x_1\in V(G_0)$ with $\chi(N_{G_0}^2(x_1))>c_{\ell-1}$, so let $N_1=N_{G_0}(x_1)$ and $G_1=G_0[N_{G_0}^2(x_1)]$.
Repeat the same to find $x_2\in V(G_1)$, $N_2=N_{G_1}(x_2)$, and $G_2=G_1[N_{G_1}^2(x_2)]$ with $\chi(G_2)>c_{\ell-2}$, and so on.
After $\ell$ such iterations and setting $C=V(G_\ell)$, it follows that $(\set{x_i}_{i\in[\ell]},\set{N_i}_{i\in[\ell]},C)$ is a multicover in $G$ with $\chi(C)>c_0$.

The multicover constructed above is far from independent, because $x_j$ has a neighbor in $N_i$ for all $i$ and $j$ with $1\leq i<j\leq\ell$.
In fact, a sufficiently long independent multicover would imply an induced $1$-subdivision of $K_{t,t}$, as stated by the next lemma.

\begin{lemma}[Chudnovsky, Scott, Seymour {\cite[7.3]{CSS21}}]
\label{lem:independent-multicover}
For all\/ $t,k\geq 1$ and\/ $c_1\geq 0$, there are\/ $m\geq 1$ and\/ $d\geq 0$ such that the following holds.
Let\/ $G$ be a graph satisfying the following conditions:
\begin{itemize}
\item $G$ contains no induced\/ $1$-subdivision of\/ $K_{t,t}$;
\item $\omega(G)\leq k$;
\item $\chi(S)\leq c_1$ for every set\/ $S\subseteq V(G)$ with\/ $\omega(S)\leq k-1$.
\end{itemize}
Then\/ $\chi(C)\leq d$ for every independent multicover\/ $(\set{x_i}_{i\in I},\set{N_i}_{i\in I},C)$ with\/ $\size{I}=m$ in\/ $G$.
\end{lemma}

We are ready to state our main technical lemma.

\begin{restatable}{lemma}{lemmain}
\label{lem:main}
For all\/ $m,k,r\geq 1$, $f\colon\setN\to\setN$, and\/ $c_1,d\geq 0$, there is\/ $c\geq 0$ such that the following holds.
Let\/ $G$ be a graph satisfying the following conditions:
\begin{itemize}
\item $G$ is\/ $(f,2)$-controlled;
\item $\omega(G)\leq k$;
\item $\beta(G)\leq r$;
\item $\chi(N(v))\leq c_1$ for every vertex\/ $v$ of\/ $G$;
\item $\chi(C)\leq d$ for every independent multicover\/ $(\set{x_i}_{i\in I},\set{N_i}_{i\in I},C)$ with\/ $\size{I}=m$ in\/ $G$.
\end{itemize}
Then\/ $\chi(G)\leq c$.
\end{restatable}

In the next subsection, we present the proof of a weaker but significantly simpler version of \cref{lem:main}.
We prove \cref{lem:main} in \cref{sec:main}.
\Cref{lem:independent-multicover,lem:main} imply \cref{thm:b-bounded} as follows.

\begin{proof}[Proof of \cref{thm:b-bounded}]
We prove that for all $t,k,r\geq 1$ and $f\colon\setN\to\setN$, there is a constant $c\geq 0$ such that every $(f,2)$-controlled graph $G$ with $\omega(G)\leq k$ and $\beta(G)\leq r$ that excludes induced $1$-subdivisions of $K_{t,t}$ satisfies $\chi(G)\leq c$.

The proof proceeds by induction on $k$.
For the base case of $k=1$, it suffices to take $c=1$.
For the induction step, suppose that $k\geq 2$ and the statement holds for $k-1$.
Thus, there is $c_1\geq 1$ such that every $(f,2)$-controlled graph $H$ with $\omega(H)\leq k-1$ and $\beta(H)\leq r$ that excludes induced $1$-subdivisions of $K_{t,t}$ satisfies $\chi(H)\leq c_1$.
For every $v\in V(G)$, since $\omega(N(v))\leq k-1$, the above applied to $H=G[N(v)]$ yields $\chi(N(v))\leq c_1$.
By \cref{lem:independent-multicover}, there are $m\geq 1$ and $d\geq 0$ such that $\chi(C)\leq d$ for every independent multicover $(\set{x_i}_{i\in I},\set{N_i}_{i\in I},C)$ with $\size{I}=m$ in $G$.
The statement for $k$ thus follows from \cref{lem:main}.
\end{proof}

\subsection{Homomorphisms from Burling graphs}

An induced subgraph of $G$ isomorphic to $B_k$ is a structurally very neat witness of the fact that $\chi(G)\geq k$, but if we just aim at finding any witness of bounded size, it is unnecessarily strong.
In this overview, we present the proof of a weaker but significantly simpler version of \cref{lem:main}, which still conveys the key ideas and, in particular, explains how Burling graphs become involved.

A \emph{homomorphism} from a graph $F$ to a graph $G$ is a mapping $\phi\colon V(F)\to V(G)$ (not necessarily injective) such that $\phi(u)\phi(v)\in E(G)$ for every $uv\in E(F)$.
A crucial property is that if there is a homomorphism $F\to G$, then $\chi(F)\leq\chi(G)$.
Thus, in order to certify that $\chi(G)\geq k$, it is enough to find a homomorphism $B_k\to G$.
Let $\gamma(G)$ denote the maximum $k$ such that there is a homomorphism $B_k\to G$; in particular, $\beta(G)\leq\gamma(G)\leq\chi(G)$ for every graph $G$.
The aforesaid weaker version of \cref{thm:b-bounded} uses a bound on $\gamma(G)$ in place of $\beta(G)$.

\begin{lemma}
\label{lem:hom-main}
For all\/ $m,k,r\geq 1$, $f\colon\setN\to\setN$, and\/ $c_1,d\geq 0$, there is\/ $c\geq 0$ such that the following holds.
Let\/ $G$ be a graph satisfying the following conditions:
\begin{itemize}
\item $G$ is\/ $(f,2)$-controlled;
\item $\omega(G)\leq k$;
\item $\gamma(G)\leq r$;
\item $\chi(N(v))\leq c_1$ for every vertex\/ $v$ of\/ $G$;
\item $\chi(C)\leq d$ for every independent multicover\/ $(\set{x_i}_{i\in I},\set{N_i}_{i\in I},C)$ with\/ $\size{I}=m$ in\/ $G$.
\end{itemize}
Then\/ $\chi(G)\leq c$.
\end{lemma}

Plugging in \cref{lem:hom-main} in place of \cref{lem:main} in the proof of \cref{thm:b-bounded} leads to the following weaker version of the latter.

\begin{theorem}
For every\/ $t\geq 1$ and every\/ $2$-controlled class\/ $\calG$ of graphs excluding induced\/ $1$-subdivisions of\/ $K_{t,t}$, there is a function\/ $g\colon\setN\times\setN\to\setN$ such that every graph\/ $G\in\calG$ satisfies\/ $\chi(G)\leq g(\omega(G),\gamma(G))$.
\end{theorem}

The rest of this subsection is devoted to the proof of \cref{lem:hom-main}.
We need several technical but important definitions and two lemmas.
We first state the lemmas and apply them to complete the proof of \cref{lem:hom-main}, and then we prove the lemmas.

For a multicover $\calN=(\set{x_i}_{i\in I},\set{N_i}_{i\in I},C)$ in $G$, an \emph{$\calN$-skewer} in $G$ is a pair $(\set{y_i}_{i\in I},z)$ with $z\in C$ and $y_i\in N_i\cap N(z)\cap\bigcap_{j\in I,j\geq i}N(x_j)$ for every $i\in I$; see \cref{fig:skewer}.
For a multicover $\calM=(\set{x_i}_{i\in[\ell]},\set{N_i}_{i\in[\ell]},C)$ and an index set $I\subseteq[\ell]$, let $\calM_I$ denote the multicover $(\set{x_i}_{i\in I},\set{N_i}_{i\in I},C)$.

\begin{figure}
  \centering
  \includegraphics[page=5,scale=0.7]{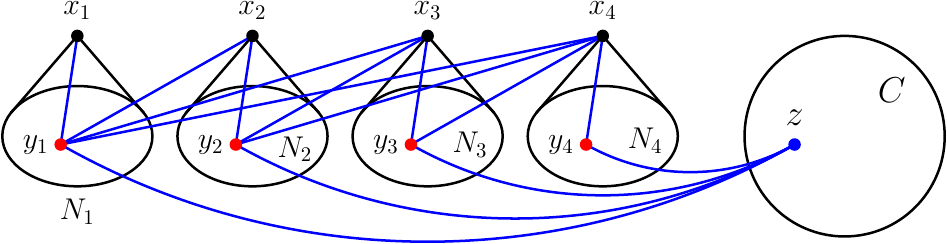}
  \caption{A multicover with a skewer}
  \label{fig:skewer}
\end{figure}

\begin{lemma}
\label{lem:hom-base}
For all\/ $m,n\geq 1$ and\/ $c_1,d\geq 0$, there are\/ $\ell,c'\geq 0$ such that the following holds.
Let\/ $G$ be a graph satisfying the following conditions:
\begin{itemize}
\item $\chi(N(v))\leq c_1$ for every vertex\/ $v$ of\/ $G$;
\item $\chi(C)\leq d$ for every independent multicover\/ $(\set{x_i}_{i\in I},\set{N_i}_{i\in I},C)$ with\/ $\size{I}=m$ in\/ $G$.
\end{itemize}
Let\/ $\calM=(\set{x_i}_{i\in[\ell]},\set{N_i}_{i\in[\ell]},C)$ be a multicover in\/ $G$ such that\/ $G$ contains no\/ $\calM_I$-skewer for any index set\/ $I\subseteq[\ell]$ with\/ $\size{I}=n$.
Then\/ $\chi(C)\leq c'$.
\end{lemma}

For a marked graph $F$ and a multicover $\calN=(\set{x_i}_{i\in I},\set{N_i}_{i\in I},C)$ in $G$, an \emph{$\calN$-homomorph} of $F$ in $G$ is a family $\set{\phi_i}_{i\in I}$ where
\begin{itemize}
\item for every $i\in I$, $\phi_i$ is a homomorphism $F\to G$ such that
\[\phi_i(V_A(F))\subseteq N_i\cap\bigcap_{j\in I,\,j\geq i}N(x_j)\qquad\text{and}\qquad\phi_i(V_N(F))\subseteq C{;}\]
\item all $\phi_i$ with $i\in I$ agree on $V_N(F)$, that is, their restrictions to $V_N(F)$ are equal.
\end{itemize}
The \emph{root} of such an $\calN$-homomorph of $F$ is the set $\bigcup_{i\in I}\phi_i(V_A(F))$.
Observe that $\set{\phi_i}_{i\in I}$ is an $\calN$-homomorph of $B'_1$ in $G$ if and only if $(\set{y_i}_{i\in I},z)$ is an $\calN$-skewer in $G$, where $y_i=\phi_i(a)$ and $z=\phi_i(b)$ for all $i\in I$ (the latter not depending on $i$, by the second condition above).
A multicover $\calM=(\set{x_i}_{i\in[\ell]},\set{N_i}_{i\in[\ell]},C)$ in $G$ is \emph{$n$-clean} if for every set $I\subseteq[\ell]$ with $1\leq\size{I}\leq n$, there is no $\calM_I$-homomorph of $B'_{n-\size{I}+1}$ in $G$.

\begin{lemma}
\label{lem:hom-step}
For all\/ $n\geq 1$, $f\colon\setN\to\setN$, and\/ $\ell,c_1,c_2\geq 0$, there is\/ $c\geq 0$ such that the following holds.
Let\/ $G$ be a graph satisfying the following conditions:
\begin{itemize}
\item there is no homomorphism\/ $B'_n\to G$;
\item $G$ is\/ $(f,2)$-controlled;
\item $\chi(N(v))\leq c_1$ for every vertex\/ $v$ of\/ $G$;
\item $\chi(C')\leq c_2$ for every\/ $n$-clean multicover\/ $(\set{x_i}_{i\in[\ell+1]},\set{N_i}_{i\in[\ell+1]},C')$ in\/ $G$.
\end{itemize}
Then\/ $\chi(C)\leq c$ for every\/ $n$-clean multicover\/ $(\set{x_i}_{i\in[\ell]},\set{N_i}_{i\in[\ell]},C)$ in\/ $G$.
\end{lemma}

\begin{proof}[Proof of \cref{lem:hom-main}]
Let $n=3\cdot 2^{r-1}-2$.
By \cref{lem:burling}~\ref{item:burling-2}, every graph $G$ with $\gamma(G)\leq r$ is $B'_n$-free.
For any $\ell,c_2\geq 0$, let $f_\ell(c_2)$ be the constant $c$ claimed by \cref{lem:hom-step} for $n,f,\ell,c_1,c_2$.
Now, let $\ell$ and $c'$ be the constants claimed by \cref{lem:hom-base} for $m,n,c_1,d$.
We show that the lemma holds with $c=(f_1\circ\cdots\circ f_\ell)(c')$.

Let $G$ be a graph as in the lemma.
We prove that for every $j\in\set{0,1,\ldots,\ell}$ and every $n$-clean multicover $(\set{x_i}_{i\in[j]},\set{N_i}_{i\in[j]},C)$ in $G$, we have $\chi(C)\leq(f_{j+1}\circ\cdots\circ f_\ell)(c')$.
The proof of the claim goes by downward induction on $j$.
For the base case of $j=\ell$, the $n$-cleanness of the multicover $\calM=(\set{x_i}_{i\in[\ell]},\set{N_i}_{i\in[\ell]},C)$ implies that $G$ contains no $\calM_I$-skewer for any set $I\subseteq[\ell]$ with $\size{I}=n$, so \cref{lem:hom-base} gives $\chi(C)\leq c'$, as required.
For the induction step with $0\leq j<\ell$, \cref{lem:hom-step} applied to the $n$-clean multicover $(\set{x_i}_{i\in[j]},\set{N_i}_{i\in[j]},C)$ with $(f_{j+2}\circ\cdots\circ f_\ell)(c')$ playing the role of $c_2$ gives $\chi(C)\leq(f_{j+1}\circ\cdots\circ f_\ell)(c')$, as required.
Now, for $j=0$ and the trivial $n$-clean multicover with empty index set and $C=V(G)$, we conclude that $\chi(G)\leq(f_1\circ\cdots\circ f_\ell)(c')=c$.
\end{proof}

\begin{proof}[Proof of \cref{lem:hom-base}]
By Ramsey's theorem, there is $\ell\geq 1$ such that for every set $Q\subseteq\smash{\binom{[\ell]}{2}}$, one of the following conditions hold:
\begin{enumeratea}
\item\label{item:hom-base-a} there is a set $I\subseteq[\ell]$ with $\size{I}=n$ and $\smash{\binom{I}{2}}\subseteq Q$;
\item\label{item:hom-base-b} there is a set $J\subseteq[\ell]$ with $\size{J}=m$ and $\smash{\binom{J}{2}}\subseteq\smash{\binom{[\ell]}{2}}\setminus Q$.
\end{enumeratea}
We show that the lemma holds with $\ell$ and $c'=\smash{\binom{\ell}{m}}d$.

Let $G$ be a graph and $\calM$ be a multicover in $G$ as in the statement of the lemma.
Consider a vertex $v\in C$.
For every $i\in[\ell]$, since $N_i$ covers $C$, there is a vertex $y_i\in N_i\cap N(v)$.
By the choice of $\ell$, one of the conditions \ref{item:hom-base-a}, \ref{item:hom-base-b} above holds for the set $Q$ of pairs $\set{i,j}\in\smash{\binom{[\ell]}{2}}$ with $i<j$ and $y_i\in N(x_j)$.
If \ref{item:hom-base-a} holds, then $(\set{y_i}_{i\in I},v)$ is an $\calM_I$-skewer, contradicting the last assumption of the lemma.
So \ref{item:hom-base-b} holds; let $J(v)$ denote the resulting set $J$.
Now, for each set $J\in\smash{\binom{[\ell]}{m}}$, let $C_J=\set{v\in C\colon J(v)=J}$.
The sets $C_J$ partition $C$.
Fix some $J\in\smash{\binom{[\ell]}{m}}$.
For each $i\in J$, let $N'_i=N_i\setminus\bigcup_{j\in J,j>i}N(x_j)$.
For every $v\in C_J$ and every $i\in J$, since $\set{\set{i,j}\colon j\in J,\:j>i}\subseteq\smash{\binom{J}{2}}\subseteq\smash{\binom{[\ell]}{2}}\setminus Q$, there is a vertex $y_i\in N_i\cap N(v)\setminus\bigcup_{j\in J,j>i}N(x_j)=N'_i\cap N(v)$.
It follows that $(\set{x_i}_{i\in J},\set{N'_i}_{i\in J},C_J)$ is an independent multicover in $G$, which implies $\chi(C_J)\leq d$.
We conclude that
\[\chi(C)\leq\sum_{J\in\binom{[\ell]}{m}}\chi(C_J)\leq\tbinom{\ell}{m}d=c'{.}\qedhere\]
\end{proof}

\begin{proof}[Proof of \cref{lem:hom-step}]
Let
\vspace*{-\medskipamount}
\[h=\sum_{k=1}^{n-1}\tbinom{\ell}{k}{,}\qquad c=2^hf(c_1+c_2+1){.}\]
Let $G$ be a graph as in the statement of the lemma, and let $\calM=(\set{x_i}_{i\in[\ell]},\set{N_i}_{i\in[\ell]},C)$ be an $n$-clean multicover in $G$.
We aim to prove that $\chi(C)\leq c$.

Let $\calH=\set{J\subseteq[\ell]\colon 1\leq\size{J}\leq n-1}$.
Thus $\size{\calH}=h$.
For each $v\in C$, let $\Phi(v)$ comprise the sets $J\in\calH$ such that there is an $\calM_J$-homomorph of $\smash[b]{B'_{n-\size{J}}}$ in $G$ whose root is a subset of $N(v)$.
For each $\Phi\subseteq\calH$, let $C_\Phi=\set{v\in C\colon\Phi(v)=\Phi}$.
The sets $C_\Phi$ partition $C$.
We prove that for every $\Phi\subseteq\calH$ and every vertex $v\in C$, we have $\chi(\smash[b]{N^2_{G[C_\Phi]}}(v))\leq c_2$.
Then, since $G$ is $(f,2)$-controlled and $\chi(N(v))\leq c_1$, we can conclude that $\chi(C_\Phi)\leq f(c_1+c_2+1)$ for every $\Phi\subseteq\calH$, which implies
\[\chi(C)\leq\sum_{\Phi\subseteq\calH}\chi(C_\Phi)\leq 2^hf(c_1+c_2+1)=c{.}\]

Fix some $\Phi\subseteq\calH$ and some vertex $x_{\ell+1}$ in $C_\Phi$.
Let $N_{\ell+1}=N_{G[C_\Phi]}(x_{\ell+1})$ and $C'=\smash[b]{N^2_{G[C_\Phi]}}(x_{\ell+1})$.
Let $\calM'=(\set{x_i}_{i\in[\ell+1]},\set{N_i}_{i\in[\ell+1]},C')$.
Since $\calM$ is a multicover in $G$, $\set{x_{\ell+1}}\cup N_{\ell+1}\cup C'\subseteq C$, $x_{\ell+1}$ is complete to $N_{\ell+1}$ and anticomplete to $C'$, and $N_{\ell+1}$ covers $C'$, it follows that $\calM'$ is a multicover in $G$.
It remains to prove that $\calM'$ is $n$-clean, because then the last assumption of the lemma implies $\chi(C')\leq c_2$, as required.

To prove that $\calM'$ is $n$-clean, we need to verify that for every set $I\subseteq[\ell+1]$ with $1\leq\size{I}\leq n$, there is no $\calM'_I$-homomorph of $\smash[b]{B'_{n-\size{I}+1}}$ in $G$.
For sets $I\subseteq[\ell]$ with $1\leq\size{I}\leq n$, the property follows from the $n$-cleanness of $\calM$, as an $\calM'_I$-homomorph is an $\calM_I$-homomorph.
For the set $I=\set{\ell+1}$, the property follows from the assumption that there is no homomorphism $B'_n\to G$.
It remains to consider subsets $I\subseteq[\ell+1]$ with $2\leq\size{I}\leq n$ and $\ell+1\in I$.
For the rest of the proof, let $I$ be such a subset.
Let $J=I\cap[\ell]$ and $k=n-\size{I}+1=n-\size{J}$.

Now comes the key point in the argument.
Suppose towards a contradiction that there is an $\calM'_I$-homomorph $\set{\phi_i}_{i\in I}$ of $B'_k$ in $G$.
It follows that the family $\set{\phi_i}_{i\in J}$, which is $\set{\phi_i}_{i\in I}$ with $\phi_{\ell+1}$ removed, is an $\calM_J$-homomorph of $B'_k$ in $G$ whose root is a subset of $N(x_{\ell+1})$ (where the latter follows from the first condition in the definition of $\calM'_I$-homomorph).
Therefore, $J\in\Phi(x_{\ell+1})=\Phi$, by the definition of $\Phi(x_{\ell+1})$ and the assumption that $x_{\ell+1}\in C_\Phi$.
Let $A=V_A(B'_k)$, and recall that the marked graph $B'_{k+1}$ is constructed from $B'_k$ by adding, for every $a\in A$, an isomorphic copy $B^a$ of $B'_k$ (on a separate vertex set) and a new vertex $b^a$ that becomes connected to $\smash[b]{N_{B'_k}}(a)\cup V_A(B^a)$ by new edges.
For every $a\in A$, let $v^a=\phi_{\ell+1}(a)$; since $v^a\in N_{\ell+1}\subseteq C_\Phi$, we have $\Phi(v^a)=\Phi\ni J$, so (since $B^a$ is isomorphic to $B'_k$) there is an $\calM_J$-homomorph $\set{\phi^a_i}_{i\in J}$ of $B^a$ in $G$ whose root is a subset of $N(v^a)$.
For every $i\in J$, let
\[\psi_i=\phi_i\cup\bigcup_{a\in A}\phi^a_i\cup\set{b^a\mapsto v^a\colon a\in A}{.}\]

We claim that $\set{\psi_i}_{i\in J}$ is an $\calM_J$-homomorph of $B'_{k+1}$ in $G$, and the inductive definition of $B'_{k+1}$ is exactly what we need to infer that.
We need to verify the following conditions:
\begin{itemize}
\item for every $i\in J$, $\psi_i$ is a homomorphism $B'_{k+1}\to G$ such that
\[\psi_i(V_A(B'_{k+1}))\subseteq N_i\cap\bigcap_{j\in J,\,j\geq i}N(x_j)\qquad\text{and}\qquad\psi_i(V_N(B'_{k+1}))\subseteq C{;}\]
\item all $\psi_i$ with $i\in J$ agree on $V_N(B'_{k+1})$, that is, their restrictions to $V_N(B'_{k+1})$ are equal.
\end{itemize}
By the definition of $\set{\psi_i}_{i\in J}$ and the fact that $V_A(B'_{k+1})=A\cup\bigcup_{a\in A}V_A(B^a)$, the respective conditions on $\set{\phi_i}_{i\in J}$ and $\set{\phi^a_i}_{i\in J}$ directly imply the conditions above except for the condition that for all $i\in J$ and $a\in A$, $\psi_i$ maps the ``new vertex'' $b^a$ into $C$ and maps the ``new edges'' connecting $b^a$ with $\smash[b]{N_{B'_k}}(a)\cup V_A(B^a)$ to edges of $G$; this is what remains to be verified.
Thus, let $i\in J$ and $a\in A$.
Since $\psi_i(b^a)=v^a=\phi_{\ell+1}(a)\in N_{\ell+1}\subseteq C$ and $\phi_{\ell+1}$ is a homomorphism $B'_k\to G$, we have $\psi_i(b^a)\in C$ and
\[\psi_i(N_{B'_k}(a))=\phi_i(N_{B'_k}(a))=\phi_{\ell+1}(N_{B'_k}(a))\subseteq N(v^a){,}\]
where the second equality holds because $A$ is an independent set and therefore $\smash[b]{N_{B'_k}}(a)\subseteq V_N(B'_k)$.
Since the root of $\set{\phi^a_i}_{i\in J}$ is a subset of $N(v^a)$, we have
\begin{gather*}
\psi_i(V_A(B^a))=\phi^a_i(V_A(B^a))\subseteq N(v^a){,}\\
\psi_i(N_{B'_{k+1}}(b^a))=\psi_i(N_{B'_k}(a))\cup\psi_i(V_A(B^a))\subseteq N(v^a){.}
\end{gather*}

We have shown that $\set{\psi_i}_{i\in J}$ is an $\calM_J$-homomorph of $B'_{k+1}$ in $G$.
This contradicts the assumption that $\calM$ is $n$-clean.
Thanks to this contradiction, we conclude that $\calM'$ is an $n$-clean multicover in $G$, which implies $\chi(C')\leq c_2$.
\end{proof}

\section{Burling graphs}
\label{sec:burling}

In this section, we prove that the same class of Burling graphs that is defined using Burling's sequence $(B_k)_{k\geq 1}$ can be defined using the sequence $(B'_k)_{k\geq 1}$ (\cref{lem:burling}) or yet a different sequence $(D_k)_{k\geq 1}$ of marked graphs that we are to define.

We start by defining some operations on marked graphs.
For two marked graphs $G$ and $H$, we let marked graphs $\Sub(G,H)$, $\Sub^*(G,H)$, and $\Ext(G,H)$ be defined as follows.
Let $A=V_A(G)$.
For each $a\in A$, let $H^a$ be an isomorphic copy of $H$ on a new vertex set, and let $b^a$ be a new vertex, so that the sets $V(G)$, $V(H^a)$ for $a\in A$, and $\set{b^a\colon a\in A}$ are pairwise disjoint.
Let
\begin{gather*}
V(\Sub(G,H))=V_N(G)\cup\bigcup_{a\in A}V(H^a){,}\qquad
V_A(\Sub(G,H))=\bigcup_{a\in A}V_A(H^a){,}\\
E(\Sub(G,H))=E_N(G)\cup\bigcup_{a\in A}\Bigl(E(H^a)\cup N_G(a)V_A(H^a)\Bigr){,} \displaybreak[0]\\
V(\Sub^*(G,H))=V(G)\cup\bigcup_{a\in A}V(H^a){,}\qquad
V_A(\Sub^*(G,H))=A\cup\bigcup_{a\in A}V_A(H^a){,}\\
E(\Sub^*(G,H))=E(G)\cup\bigcup_{a\in A}\Bigl(E(H^a)\cup N_G(a)V_A(H^a)\Bigr){,} \displaybreak[0]\\
V(\Ext(G,H))=V(G)\cup\bigcup_{a\in A}V(H^a)\cup\set{b^a\colon a\in A}{,}\qquad
V_A(\Ext(G,H))=A\cup\bigcup_{a\in A}V_A(H^a){,}\\
E(\Ext(G,H))=E(G)\cup\bigcup_{a\in A}\Bigl(E(H^a)\cup N_G(a)\set{b^a}\cup\set{b^a}V_A(H^a)\Bigr){.}
\end{gather*}
It is clear that $\Sub(G,H)\subseteq\Sub^*(G,H)$, and if $G\subseteq G'$ and $H\subseteq H'$, then $\Sub(G,H)\subseteq\Sub(G',H')$, $\Sub^*(G,H)\subseteq\Sub^*(G',H')$, and $\Ext(G,H)\subseteq\Ext(G',H')$.
Also, $B'_{k+1}=\Ext(B'_k,B'_k)$ for $k\geq 1$, by definition.

\begin{figure}
  \begin{subfigure}{0.48\textwidth}
    \centering
    \includegraphics[page=3,scale=0.6]{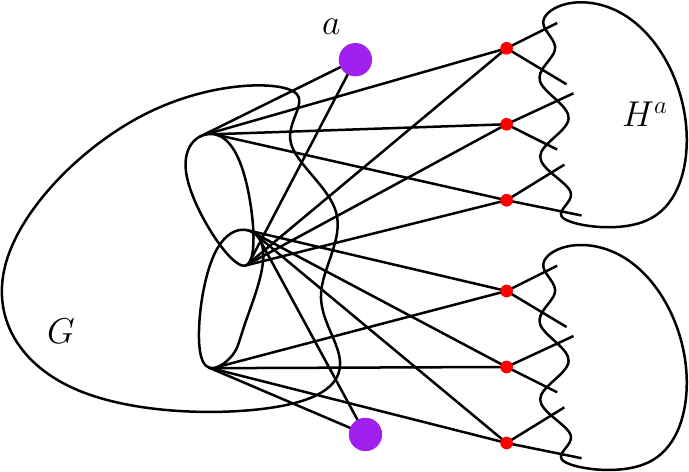}
    \caption{$\Sub(G,H)$ and $\Sub^*(G,H)$}
  \end{subfigure}
  \begin{subfigure}{0.48\textwidth}
    \centering
    \includegraphics[page=4,scale=0.6]{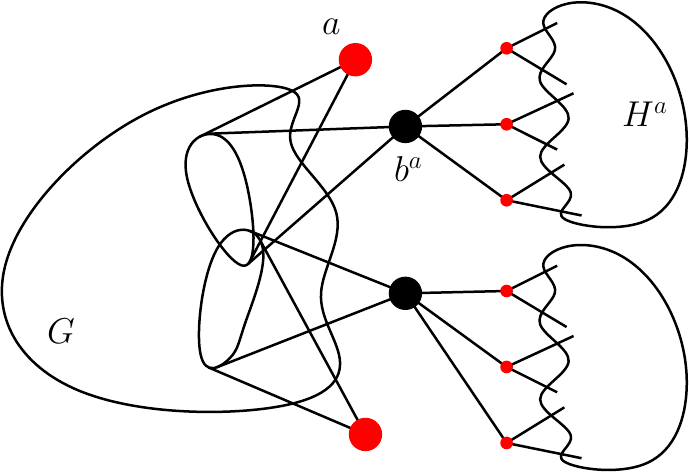}
    \caption{$\Ext(G,H)$}
  \end{subfigure}
  \caption{
    Schematic depiction of the operations on marked graphs.
    Vertices in $V_A(\Sub(G,H))$ and $V_A(\Ext(G,H))$ are depicted in red.
    Vertices in $V_A(\Sub^*(G,H))$ are the red and the purple ones, where the latter are not vertices of $\Sub(G,H)$.
  }
\end{figure}

The next four lemmas follow easily from the definitions above.

\begin{lemma}
\label{lem:burling-subext}
For\/ $k\geq 2$, we have\/ $B_{k+1}=\Sub(B_k,\Ext(B_k,B_1))$.
\end{lemma}

\begin{proof}
Let $G=B_k$ and $A=V_A(G)$.
For each $a\in A$, let $H^a$ be an isomorphic copy of $B_k$ on a new vertex set, and let $A^a=V_A(H^a)$.
For each $a\in A$ and each $v\in A^a$, let $b^a_v$ and $c^a_v$ be two new vertices.
For every $a\in A$, we have
\begin{gather*}
V(\Ext(H^a,B_1))=V(H^a)\cup\set{b^a_v,c^a_v\colon v\in A^a}{,}\qquad
V_A(\Ext(H^a,B_1))=A^a\cup\set{c^a_v\colon v\in A^a}{,}\\
E(\Ext(H^a,B_1))=E(H^a)\cup\bigcup_{v\in A^a}\Bigl(N_{H^a}(v)\set{b^a_v}\cup\set{b^a_vc^a_v}\Bigr){.}
\end{gather*}
By definition, we have
\begin{gather*}
\begin{aligned}
V(B_{k+1})&=V(G-A)\cup\bigcup_{a\in A}\Bigl(V(H^a)\cup\set{b^a_v,c^a_v\colon v\in A^a}\Bigr)\\
&=V(G-A)\cup\bigcup_{a\in A}V(\Ext(H^a,B_1))=V\bigl(\Sub(B_k,\Ext(B_k,B_1))\bigr){,}
\end{aligned}\\
V_A(B_{k+1})=\bigcup_{a\in A}\Bigl(A^a\cup\set{c^a_v\colon v\in A^a}\Bigr)=\bigcup_{a\in A}V_A(\Ext(H^a,B_1))=V_A\bigl(\Sub(B_k,\Ext(B_k,B_1))\bigr){,}\\
\begin{aligned}
E(B_{k+1})&=E(G-A)\cup\bigcup_{a\in A}\biggl(E(H^a)\cup\bigcup_{v\in A^a}\Bigl(\set{b^a_vc^a_v}\cup N_{H^a}(v)\set{b^a_v}\cup N_G(a)\set{v,c^a_v}\Bigr)\biggr)\\
&=E(G-A)\cup\bigcup_{a\in A}\Bigl(E(\Ext(H^a,B_1))\cup N_G(a)V_A(\Ext(H^a,B_1))\Bigr)\\
&=E\bigl(\Sub(B_k,\Ext(B_k,B_1))\bigr){.}
\end{aligned}
\end{gather*}
Thus $B_{k+1}=\Sub(B_k,\Ext(B_k,B_1))$.
\end{proof}

\begin{lemma}
\label{lem:ext}
For any marked graphs\/ $G$ and\/ $H$, if\/ $V_A(H)\neq\emptyset$, then\/ $\Ext(G,H)\subseteq\Sub(\Ext(G,B_1),H)$.
\end{lemma}

\begin{proof}
For each $a\in A$, let $b^a$ and $c^a$ be two new vertices.
We have
\begin{gather*}
V(\Ext(G,B_1))=V(G)\cup\set{b^a,c^a\colon a\in A}{,}\qquad
V_A(\Ext(G,B_1))=A\cup\set{c^a\colon a\in A}{,}\\
E(\Ext(G,B_1))=E(G)\cup\bigcup_{a\in A}\Bigl(N_G(a)\set{b^a}\cup\set{b^ac^a}\Bigr){.}
\end{gather*}
For each $a\in A$, let $H^a$ be an isomorphic copy of $H$ on a new vertex set, and let $V^a=V_A(H^a)$.
Since $V_A(\Ext(G,B_1))=A\cup\set{c^a\colon a\in A}$, the construction of $\Sub(\Ext(G,B_1),H)$ involves $2\size{A}$ isomorphic copies of $H$: one for each vertex $c^a$ with $a\in A$, which we take to be $H^a$, and one for each vertex $a\in A$, say $\hat H^a$, for which we just identify $a$ with an arbitrary vertex in $V_A(\hat H^a)$ making use of the assumption that $V_A(H)\neq\emptyset$.
Thus, we have
\begin{gather*}
\begin{aligned}
V(\Ext(G,H))&=V(G)\cup\set{b^a\colon a\in A}\cup\bigcup_{a\in A}V(H^a)\\
&=V_N(\Ext(G,B_1))\cup A\cup\bigcup_{a\in A}V(H^a)\subseteq V\bigl(\Sub(\Ext(G,B_1),H)\bigr){,}
\end{aligned}\displaybreak[0]\\
V_A(\Ext(G,H))=A\cup\bigcup_{a\in A}A^a=V_A\bigl(\Sub(\Ext(G,B_1),H)\bigr)\cap V(\Ext(G,H)){,}\displaybreak[0]\\
\begin{aligned}
E(\Ext(G,H))&=E(G)\cup\bigcup_{a\in A}N_G(a)\set{b^a}\cup\bigcup_{a\in A}\Bigl(E(H^a)\cup\set{b^a}A^a\Bigr)\\
&=E_N(\Ext(G,B_1))\cup\bigcup_{a\in A}N_G(a)\set{a}\cup\bigcup_{a\in A}\Bigl(E(H^a)\cup\set{b^a}A^a\Bigr)\\
&=E\bigl(\,\Sub(\Ext(G,B_1),H)\,\bigl[\,V(\Ext(G,H))\,\bigr]\,\bigr){.}
\end{aligned}
\end{gather*}
Thus $\Ext(G,H)$ is an induced subgraph of $\Sub(\Ext(G,B_1),H)$.
\end{proof}

\begin{lemma}
\label{lem:sub*-base}
For every marked graph\/ $K$, we have\/ $\Sub^*(B_2,K)\subseteq\Ext(B_2,\Ext(B_2,K))$.
\end{lemma}

\begin{proof}
Let $a,b\notin V(K)$.
Up to isomorphism, the subgraph $K'$ of $\Ext(B_2,K)$ induced on $V(K)\cup\set{a}$ has $V_A(K')=V_A(K)\cup\set{a}$ and $E(K')=E(K)$.
Also up to isomorphism, we have
\begin{gather*}
V(\Sub^*(B_2,K))=V(K)\cup\set{a,b}=V(K')\cup\set{b}{,}\qquad
V_A(\Sub^*(B_2,K))=V_A(K)\cup\set{a}=V_A(K'){,}\\
E(\Sub^*(B_2,K))=\set{ba}\cup E(K)\cup\set{b}V_A(K)=E(K')\cup\set{b}V_A(K'){.}
\end{gather*}
It follows that $\Sub^*(B_2,K)$ is the subgraph of $\Ext(B_2,K')$ induced on $V(K')\cup\set{b}$, so it is an induced subgraph of $\Ext(B_2,\Ext(B_2,K))$.
\end{proof}

\begin{lemma}
\label{lem:sub*-ind}
For any marked graphs\/ $G$, $H$, and\/ $K$, we have
\[\Sub^*(\Ext(G,H),K)\subseteq\Ext(\Sub^*(G,K),\:\Sub^*(H,K)){.}\]
\end{lemma}

\begin{proof}
Let $A=V_A(G)$.
For each $a\in A$, let $H^a$ and $K^a$ be isomorphic copies of $H$ and $K$, respectively, on new vertex sets, let $A^a=V_A(H^a)$, and let $b^a$ be a new vertex.
For each $a\in A$ and $v\in V^a$, let $K^a_v$ be an isomorphic copy of $K$ on a new vertex set.
Let
\[H^*=\bigcup_{a\in A}H^a{,}\qquad K^*=\bigcup_{a\in A}K^a{,}\qquad K^a_*=\bigcup_{v\in A^a}K^a_v\quad\text{for each }a\in A{,}\qquad K^*_*=K^*\cup\bigcup_{a\in A}K^a_*{.}\]
Now, $\Ext(G,H)$ is the graph with
\begin{gather*}
V(\Ext(G,H))=V(G)\cup V(H^*)\cup\set{b^a\colon a\in A}{,}\qquad
V_A(\Ext(G,H))=A\cup V_A(H^*){,}\\
E(\Ext(G,H))=E(G)\cup E(H^*)\cup\bigcup_{a\in A}\Bigl(N_G(a)\set{b^a}\cup\set{b^a}A^a\Bigr){,}
\end{gather*}
$\Sub^*(G,K)$ is the graph with
\begin{gather*}
V(\Sub^*(G,K))=V(G)\cup V(K^*){,}\qquad
V_A(\Sub^*(G,K))=A\cup V_A(K^*){,}\\
E(\Sub^*(G,K))=E(G)\cup E(K^*)\cup\bigcup_{a\in A}N_G(a)V_A(K^a){,}
\end{gather*}
and for each $a\in A$, $\Sub^*(H^a,K)$ is the graph with
\begin{gather*}
V(\Sub^*(H^a,K))=V(H^a)\cup V(K^a_*){,}\qquad
V_A(\Sub^*(H^a,K))=A^a\cup V_A(K^a_*){,}\\
E(\Sub^*(H^a,K))=E(H^a)\cup E(K^a_*)\cup\bigcup_{v\in A^a}N_{H^a}(v)V_A(K^a_v){.}
\end{gather*}
Now, since $A\subseteq V_A(\Sub^*(G,K))$, we have
\begin{align*}
V\bigl(\Sub^*(\Ext(G,H),K)\bigr)&=V(\Ext(G,H))\cup V(K^*_*)=V(G)\cup V(H^*)\cup\set{b^a\colon a\in A}\cup V(K^*_*)\\
&=V(\Sub^*(G,K))\cup\bigcup_{a\in A}V(\Sub^*(H^a,K))\cup\set{b^a\colon a\in A}\\
&\subseteq V\bigl(\Ext(\Sub^*(G,K),\:\Sub^*(H,K))\bigr){.}
\end{align*}
We also have
\begin{align*}
V_A\bigl(\Sub^*(\Ext(G,H),K)\bigr)&=V_A(\Ext(G,H))\cup V_A(K^*_*)=A\cup V_A(H^*)\cup V_A(K^*_*)\\
&=V_A(\Sub^*(G,K))\cup\bigcup_{a\in A}V_A(\Sub^*(H^a,K))\\
&=V_A\bigl(\Ext(\Sub^*(G,K),\:\Sub^*(H,K))\bigr)\cap V\bigl(\Sub^*(\Ext(G,H),K)\bigr){.}
\end{align*}
Finally, since $N_{\Ext(G,H)}(v)=N_{H^a}(v)\cup\set{b^a}$ for each $a\in A$ and each $v\in A^a$, we have
\begin{align*}
&E\bigl(\Sub^*(\Ext(G,H),K)\bigr)\\
&\quad=E(\Ext(G,H))\cup E(K^*_*)\cup\bigcup_{a\in A}\biggl(N_{\Ext(G,H)}(a)V_A(K^a)\cup\bigcup_{v\in A^a}N_{\Ext(G,H)}(v)V_A(K^a_v)\biggr)\\
&\quad=E(G)\cup E(H^*)\cup E(K^*_*)\\
&\quad\qquad\cup\bigcup_{a\in A}\biggl(N_G(a)\set{b^a}\cup\set{b^a}A^a\cup N_G(a)V_A(K^a)\cup\bigcup_{v\in A^a}\bigl(N_{H^a}(v)\cup\set{b^a}\bigr)V_A(K^a_v)\biggr)\\
&\quad=E(\Sub^*(G,K))\cup\bigcup_{a\in A}\Bigl(E(\Sub^*(H^a,K))\cup N_G(a)\set{b^a}\cup\set{b^a}\bigl(A^a\cup V_A(K^a_*)\bigr)\Bigr)\\
&\quad=E\bigl(\,\Ext(\Sub^*(G,K),\:\Sub^*(H,K))\,\bigl[\,V(\Sub^*(\Ext(G,H),K))\,\bigr]\,\bigr){.}
\end{align*}
This shows that $\Sub^*(\Ext(G,H),K)$ is an induced subgraph of $\Ext(\Sub^*(G,K),\:\Sub^*(H,K))$.
\end{proof}

The main induction in the proof of \cref{lem:burling} is covered by the following lemma.

\begin{lemma}
\label{lem:sub*-burling}
For all\/ $m,n\geq 1$, we have\/ $\Sub^*(B'_m,B'_n)\subseteq B'_{m+n+1}$.
\end{lemma}

\begin{proof}
We prove this by induction on $m$.
For the base case of $m=1$, by \cref{lem:sub*-base}, we have
\[\Sub^*(B'_1,B'_n)\subseteq\Ext(B'_1,\Ext(B'_1,B'_n))\subseteq\Ext(B'_{n+1},\Ext(B'_n,B'_n))=B'_{n+2}{.}\]
For the induction step, suppose $m\geq 1$ and $\Sub^*(B'_m,B'_n)\subseteq B'_{m+n+1}$.
By \cref{lem:sub*-ind}, we have
\begin{align*}
\Sub^*(B'_{m+1},B'_n)&=\Sub^*(\Ext(B'_m,B'_m),B'_n)\subseteq\Ext(\Sub^*(B'_m,B'_n),\:\Sub^*(B'_m,B'_n))\\
&\subseteq\Ext(B'_{m+n+1},B'_{m+n+1})=B'_{m+n+2}{.}\qedhere
\end{align*}
\end{proof}

We are ready to prove \cref{lem:burling}, which we restate below.

\lemburling*

\begin{proof}
First, we prove \ref{item:burling-1} by induction on $k$.
We have $B'_1=B_2$ by definition.
Now, suppose $k\geq 1$ and $B'_k\subseteq B_{2k}$.
By \cref{lem:burling-subext} and the fact that $G\subseteq\Ext(G,H)$ and $G,H\subseteq\Sub(G,H)$ whenever $V_A(G)\neq\emptyset$ and $V_A(H)\neq\emptyset$, we have
\[B'_k\subseteq B_{2k}\subseteq B_{2k+1}\subseteq\Ext(B_{2k+1},B_1){,}\qquad\Ext(B'_k,B_1)\subseteq\Ext(B_{2k},B_1)\subseteq B_{2k+1}{.}\]
Consequently, by \cref{lem:ext}, we have
\[B'_{k+1}=\Ext(B'_k,B'_k)\subseteq\Sub(\Ext(B'_k,B_1),B'_k)\subseteq\Sub(B_{2k+1},\Ext(B_{2k+1},B_1))=B_{2k+2}{.}\]
This completes the induction step.

Now, we prove \ref{item:burling-2} by induction on $k$.
We have $B_2=B'_1$ by definition.
Now, suppose $k\geq 2$ and $B_k\subseteq\smash[b]{B'_{3\cdot 2^{k-2}-2}}$.
We have
\[\Ext(B_k,B_1)\subseteq\Ext(B'_{3\cdot 2^{k-2}-2},B'_{3\cdot 2^{k-2}-2})=B'_{3\cdot 2^{k-2}-1}{.}\]
By \cref{lem:burling-subext,lem:sub*-burling}, we have
\[B_{k+1}=\Sub(B_k,\Ext(B_k,B_1))\subseteq\Sub(B'_{3\cdot 2^{k-2}-2},B'_{3\cdot 2^{k-2}-1})\subseteq\Sub^*(B'_{3\cdot 2^{k-2}-2},B'_{3\cdot 2^{k-2}-1})\subseteq B'_{3\cdot 2^{k-1}-2}{.}\]
This completes the induction step.
\end{proof}

For the purpose of proving our main result, we need another (simpler) sequence $(D_k)_{k\geq 1}$ of marked graphs such that the induced subgraphs of the graphs $D_k$ are exactly the Burling graphs.
To define it, we use one more operation on marked graphs.
For a marked graph $G$ and a vertex $a^*\in V_A(G)$, let a marked graph $\Dup(G,a^*)$ be defined as follows.
Let $G'$ be an isomorphic copy of $G$ on a new vertex set, and let $b^*$ be a new vertex.
Let
\begin{gather*}
V(\Dup(G,a^*))=V(G)\cup V(G')\cup\set{b^*}{,}\qquad
V_A(\Dup(G,a^*))=V_A(G)\cup V_A(G'){,}\\
E(\Dup(G,a^*))=E(G)\cup E(G')\cup N_G(a^*)\set{b^*}\cup\set{b^*}V_A(G'){.}
\end{gather*}
Observe that $\Dup(G,a^*)\subseteq\Ext(G,G)$ after identifying $b^*$ and $G'$ with the vertex $b^{a^*}$ and with the isomorphic copy $H^{a^*}$ of $G$ occurring in the definition of $\Ext(G,G)$.

\begin{lemma}
\label{lem:alt-burling}
There exist an increasing sequence\/ $(D_k)_{k\geq 1}$ of connected marked graphs and a sequence\/ $(a^*_k)_{k\geq 1}$ with\/ $a^*_k\in V_A(D_k)$ for all\/ $k\geq 1$ and with the following properties:
\begin{enumeratei}
\item\label{item:alt-1} $D_1=B'_1=B_2$;
\item\label{item:alt-2} for every\/ $k\geq 1$, $D_{k+1}=\Dup(D_k,a^*_k)$;
\item\label{item:alt-3} for every\/ $k\geq 1$, $D_k\subseteq B'_k$;
\item\label{item:alt-4} for every\/ $k\geq 1$, there is\/ $\ell\geq 1$ such that\/ $B'_k\subseteq D_\ell$;
\item\label{item:alt-5} for every\/ $k\geq 1$, $\size{V_A(D_k)}=2^{k-1}$ and\/ $\size{V_N(D_k)}=2^k-1$.
\end{enumeratei}
\end{lemma}

\begin{proof}
Let $(m_k)_{k\geq 1}$ be the integer sequence defined inductively as follows: $m_1=1$ and $m_{k+1}=m_k+\size{V_A(B'_k)}$ for all $k\geq 1$.
We define the sequences $(D_k)_{k\geq 1}$ and $(a^*_k)_{k\geq 1}$ inductively in rounds, making sure that by the beginning of every round $r\geq 1$, we have defined $D_1,\ldots,D_{m_r}$ and $a^*_1,\ldots,a^*_{m_r-1}$ so that they satisfy the following conditions:
\begin{enumeratei'}
\item\label{item:alt-1'} $D_1=B'_1=B_2$;
\item\label{item:alt-2'} for every $k\in[m_r-1]$, $D_{k+1}=\Dup(D_k,a^*_k)$;
\item\label{item:alt-3'} for every $k\in[m_r]$, $D_k\subseteq B'_k$;
\item\label{item:alt-4'} for every $k\in[r]$, $B'_k\subseteq D_{m_k}$.
\end{enumeratei'}
Statements \ref{item:alt-1}--\ref{item:alt-4} then follow directly from \ref{item:alt-1'}--\ref{item:alt-4'} while \ref{item:alt-5} follows from \ref{item:alt-1}, \ref{item:alt-2}, and the definition of $\Dup$ by straightforward induction.

Let $D_1=B'_1=B_2$, as required in \ref{item:alt-1'}.
In round $r\geq 1$, we define $D_{m_r+1},\ldots,D_{m_{r+1}}$ and $a^*_{m_r},\ldots,a^*_{m_{r+1}-1}$.
Enumerate the vertices in $V_A(B'_r)$ as $a_0,\ldots,a_{s-1}$ where $s=\size{V_A(B'_r)}=m_{r+1}-m_r$.
For $i\in\set{0,\ldots,s-1}$, by induction, we let $a^*_{m_r+i}=a_i$, which is a vertex in $V_A(D_{m_r+i})$ as $B'_r\subseteq D_{m_r}\subseteq D_{m_r+i}$, and we let $D_{m_r+i+1}=\Dup(D_{m_r+i},a_i)$.
This makes \ref{item:alt-2'} hold for the next round.
We also have the following, by induction:
\[D_{m_r+i+1}=\Dup(D_{m_r+i},a^*_{m_r+i})\subseteq\Ext(D_{m_r+i},D_{m_r+i})\subseteq\Ext(B'_{m_r+i},B'_{m_r+i})=B'_{m_r+i+1}{.}\]
So \ref{item:alt-3'} holds for the next round.

To prove \ref{item:alt-4'} for the next round, it suffices to show that $B'_{r+1}\subseteq D_{m_{r+1}}$.
For every $i\in\set{0,\ldots,s-1}$, let $G'_i$ and $b_i$ denote the isomorphic copy $G'$ of $D_{m_r+i}$ and the vertex $b^*$ used in the definition of $D_{m_r+i+1}$ as $\Dup(D_{m_r+i},a_i)$; also, let $H_i$ be an isomorphic copy of $B'_r$ identified as an induced marked subgraph of $G'_i$, which exists as $B'_r\subseteq D_{m_r}\subseteq D_{m_r+i}$.
Assume that the definition of $B'_{r+1}$ as $\Ext(B'_r,B'_r)$ uses $H_i$ as the isomorphic copy $H^{a_i}$ of $B'_r$ and uses $b_i$ as the vertex $b^{a_i}$ for every $i\in\set{0,\ldots,s-1}$.
We prove that for every $j\in\set{0,\ldots,s}$, $D_{m_r+j}$ contains the following graph $F_j$ as an induced marked subgraph:
\begin{gather*}
V(F_j)=V(B'_r)\cup\bigcup_{i=0}^{j-1}\bigl(V(H_i)\cup\set{b_i}\bigr){,}\qquad V_A(F_j)=V_A(B'_r)\cup\bigcup_{i=0}^{j-1}V_A(H_i){,}\\
E(F_j)=E(B'_r)\cup\smash[t]{\bigcup_{i=0}^{j-1}}\Bigl(E(H_i)\cup N_{B'_r}(a_i)\set{b_i}\cup\set{b_i}V_A(H_i)\Bigr){.}
\end{gather*}
This follows by induction on $j$: we have $F_0=B'_r\subseteq D_{m_r}$, and for every $j\in\set{0,\ldots,s-1}$, if $F_j\subseteq D_{m_r+j}$, then $F_{j+1}\subseteq D_{m_r+j+1}$, by the fact that $H_j\subseteq G'_j$ and the definition of $D_{m_r+j+1}$ as $\Dup(D_{m_r+j},a_j)$ which uses $G'_j$ and $b_j$ as $G'$ and $b^*$.
In particular, the claim holds for $j=s$, so $B'_{r+1}=F_s\subseteq D_{m_r+s}=D_{m_{r+1}}$.
\end{proof}

\section{Main result}
\label{sec:main}

This section is devoted to the proof of \cref{lem:main}, which then implies all our Burling-control results.
Let $(D_k)_{k\geq 1}$ and $(a^*_k)_{k\geq 1}$ be sequences claimed by \cref{lem:alt-burling}.
By \cref{lem:burling,lem:alt-burling}, every graph from Burling's sequence is an induced subgraph of $D_k$ for sufficiently large $k$.
The proof uses the graphs $D_k$ in a similar way as the proof of \cref{lem:hom-main} uses the graphs $B'_k$.

For a clique $X$ in a graph $G$, let $N^1(X)$ denote the set of vertices in $V(G)\setminus X$ that are complete to $X$, and let $N^2(X)$ denote the set of vertices in $V(G)\setminus(X\cup N^1(X))$ that are anticomplete to $X$ and have a neighbor in $N^1(X)$.
(If $N^1(X)=\emptyset$, then $N^2(X)=\emptyset$.)

Let $s\geq 1$ and $f\colon\setN\to\setN$.
A graph $G$ is \emph{$(f,s)$-clique-controlled} if the following holds for every induced subgraph $H$ of $G$: for every $c\geq 0$, if every $s$-clique $X$ in $H$ satisfies $\chi(N_H^2(X))\leq c$, then $\chi(H)\leq f(c)$.
An \emph{$s$-clique-multicover} in $G$ is a sequence
\[\calM=(W_0,\;X_1,N_{1,1},W_1,\;X_2,N_{1,2},N_{2,2},W_2,\;\ldots,\;X_\ell,N_{1,\ell},\ldots,N_{\ell,\ell},W_\ell)\]
of subsets of $V(G)$ with $W_0=V(G)$ and with the following properties in $G$:
\begin{itemize}
\item for every $i\in[\ell]$, $X_i$ is an $s$-clique complete to $N_{i,i}$ and anticomplete to $W_i$;
\item for every $i\in[\ell]$, $X_i\cup N_{i,i}\cup W_i\subseteq W_{i-1}$;
\item for every $i\in[\ell]$, $N_{i,i}\supseteq N_{i,i+1}\supseteq\cdots\supseteq N_{i,\ell}$;
\item for all $i$ and $j$ with $1\leq i<j\leq\ell$, $N_{i,j}$ covers $W_j$.
\end{itemize}
For such an $s$-clique-multicover $\calM$, the number $\ell$ is the \emph{length} of $\calM$, and $W(\calM)$ denotes the set $W_\ell$.
A pair of indices $i$ and $j$ with $1\leq i<j\leq\ell$ is
\begin{itemize}
\item\emph{independent} in $\calM$ if $N_{i,j}$ is anticomplete to some vertex in $X_j$;
\item\emph{skew} in $\calM$ if $N_{i,j}$ is complete to $X_j$, and $N_{i,j-1}\setminus N_{i,j}$ is anticomplete to $W_j$.
\end{itemize}
An $s$-clique-multicover $\calM$ of length $\ell$ is \emph{tidy} if every pair of indices $i$ and $j$ with $1\leq i<j\leq\ell$ is independent or skew in $\calM$.
Given a tidy $s$-clique-multicover $\calM$ of length $\ell$, let $\calS(\calM)$ denote the set of non-empty subsets $I$ of $[\ell]$ such that all pairs $i,j\in I$ with $i<j$ are skew in $\calM$.

An \emph{embedding} of a graph $F$ in $G$ is an injective map $\phi\colon V(F)\to V(G)$ such that for any distinct $a,b\in V(F)$, we have $\phi(a)\phi(b)\in E(G)$ if and only if $ab\in E(F)$.
The last property of $s$-clique-multicovers that we need concerns embeddings of the graphs $D_k$ in $G$.
For an integer $k\geq 1$, a family $\set{N_i}_{i\in I}$ of pairwise disjoint subsets of $V(G)$ (where $I$ is a non-empty index set), and a set $W\subseteq V(G)$ disjoint from each set in $\set{N_i}_{i\in I}$, a \emph{$(\set{N_i}_{i\in I},W)$-embedding} of $D_k$ in $G$ is a family $\set{\phi_i}_{i\in I}$ where
\begin{itemize}
\item $\phi_i$ is an embedding of $D_k$ in $G[N_i\cup W]$ such that $\phi_i(V_A(D_k))\subseteq N_i$ and $\phi_i(V_N(D_k))\subseteq W$, for every $i\in I$;
\item all $\phi_i$ with $i\in I$ agree on $V_N(D_k)$, that is, their restrictions to $V_N(D_k)$ are equal;
\item $\bigcup_{i\in I}\phi_i(V_A(D_k))$ is an independent set in $G$.
\end{itemize}
An $s$-clique-multicover $\calM$ is \emph{$n$-tidy} if it is tidy and for every set $I\in\calS(\calM)$ with $1\leq\size{I}\leq n$, there is no $(\set{N_{i,\max I}}_{i\in I},W_{\max I})$-embedding of $D_{n-\size{I}+1}$ in $G$ (using the notation from the definition of $s$-clique-multicover).

\begin{lemma}
\label{lem:tidy-base}
Let\/ $s,n\geq 1$ and\/ $c_1,c_2\geq 0$.
Let\/ $G$ be a graph such that
\begin{itemize}
\item $\chi(N(v))\leq c_1$ for every vertex\/ $v$ of\/ $G$;
\item $\chi(N^2(X))\leq c_2$ for every\/ $(s+1)$-clique\/ $X$ in\/ $G$.
\end{itemize}
Let\/ $\calM$ be an\/ $n$-tidy\/ $s$-clique-multicover in\/ $G$ such that\/ $\max_{I\in\calS(\calM)}\size{I}>n$.
Then\/ $\chi(W(\calM))\leq(n-1)(c_1+c_2)$.
\end{lemma}

\begin{proof}
Let
\[\calM=(W_0,\;X_1,N_{1,1},W_1,\;X_2,N_{1,2},N_{2,2},W_2,\;\ldots,\;X_\ell,N_{1,\ell},\ldots,N_{\ell,\ell},W_\ell){.}\]
By assumption, there is $I\in\calS(\calM)$ with $\size{I}=n+1$.
Let $k=\max I$ and $J=I\setminus\set{k}$, so that $\size{J}=n$.
For every $v\in\bigcup_{i\in J}N_{i,k}$, the set $X_k\cup\set{v}$ is an $(s+1)$-clique in $G$, which implies $\chi(N^2(X_k\cup\set{v}))\leq c_2$.
We claim that there is a set $U\subseteq\bigcup_{i\in J}N_{i,k}$ with $\size{U}\leq n-1$ and $W_\ell\subseteq\bigcup_{v\in U}(N(v)\cup N^2(X_k\cup\set{v}))$, which then implies
\[\chi(W_\ell)\leq\sum_{v\in U}\bigl(\chi(N(v))+\chi(N^2(X_k\cup\set{v}))\bigr)\leq\size{U}\cdot(c_1+c_2)\leq(n-1)(c_1+c_2){.}\]

Suppose not.
Recall that $V_A(D_1)=\set{a}$ and $V_N(D_1)=\set{b}$.
We proceed to find an independent set $\set{v_i}_{i\in J}$ in $G$ with $v_i\in N_{i,k}$ for every $i\in J$.
Then, for any $x\in X_k$, since the set $\bigcup_{i\in J}N_{i,k}$ is complete to $x$, the family of maps $\set{\phi_i}_{i\in J}$ with $\phi_i=\set{a\mapsto v_i,\:b\mapsto x}$ for all $i\in J$ is a $(\set{N_{i,\max J}}_{i\in J},W_{\max J})$-embedding of $D_1$ in $G$, contradicting the assumption that $\calM$ is $n$-tidy.

We find the vertices $v_i$ in the decreasing order of the indices $i\in J$, as follows: for $i\in J$, assuming that we have already found $v_j$ for every index $j\in J$ with $i<j$, we choose an arbitrary vertex $v_i\in N_{i,k}\setminus\bigcup_{i<j\in J}N(v_j)$.
To this end, we need to argue that the latter set is non-empty.
We have $N_{i,k}\cap N(v_j)\subseteq N^1(X_k\cup\set{v_j})$ for all $j\in J$ with $i<j$.
Since $W_\ell$ is anticomplete to $X_k$, the set of vertices in $W_\ell$ covered by $N_{i,k}\cap\bigcup_{i<j\in J}N(v_j)$ is contained in $\bigcup_{i<j\in J}(N(v_j)\cup N^2(X_k\cup\set{v_j}))$.
By our assumption, $W_\ell$ is not a subset of the latter, so it contains a vertex covered by $N_{i,k}\setminus\bigcup_{i<j\in J}N(v_j)$, showing that the latter set is indeed non-empty.
\end{proof}

\begin{lemma}
\label{lem:tidy-bound}
For all\/ $m,n,s\geq 1$ and\/ $c_1,c_2,d\geq 0$, there are\/ $\ell,c'\geq 0$ such that the following holds.
Let\/ $G$ be a graph satisfying the following conditions:
\begin{itemize}
\item $\chi(N(v))\leq c_1$ for every vertex\/ $v$ of\/ $G$;
\item $\chi(N^2(X))\leq c_2$ for every\/ $(s+1)$-clique\/ $X$ in\/ $G$;
\item $\chi(C)\leq d$ for every independent multicover\/ $(\set{x_i}_{i\in I},\set{N_i}_{i\in I},C)$ with\/ $\size{I}=m$ in\/ $G$.
\end{itemize}
Then\/ $\chi(W(\calM))\leq c'$ for every\/ $n$-tidy\/ $s$-clique-multicover\/ $\calM$ of length\/ $\ell$ in\/ $G$.
\end{lemma}

\begin{proof}
By Ramsey's theorem, there is $\ell\geq 1$ such that for every partition of $\smash{\binom{[\ell]}{2}}$ into subsets $Q_0,Q_1,\ldots,Q_s$, one of the following conditions hold:
\begin{enumeratea}
\item\label{item:tidy-bound-a} there is a set $I\subseteq[\ell]$ with $\size{I}=n+1$ and $\smash{\binom{I}{2}}\subseteq Q_0$;
\item\label{item:tidy-bound-b} there are $r\in[s]$ and a set $J\subseteq[\ell]$ with $\size{J}=m$ and $\smash{\binom{J}{2}}\subseteq Q_r$.
\end{enumeratea}
We show that the lemma holds with $\ell$ and $c'=\max((n-1)(c_1+c_2),d)$.

Let $G$ be a graph as in the statement of the lemma.
Consider an $n$-tidy $s$-clique-multicover
\[\calM=(W_0,\;X_1,N_{1,1},W_1,\;X_2,N_{1,2},N_{2,2},W_2,\;\ldots,\;X_\ell,N_{1,\ell},\ldots,N_{\ell,\ell},W_\ell)\]
in $G$.
Let $X_i=\set{x_{i,1},\ldots,x_{i,s}}$ for every $i\in[\ell]$.
Let $Q_0$ be the set of pairs $\set{i,j}\in\smash{\binom{[\ell]}{2}}$ that are skew in $\calM$.
For every $r\in[s]$, let $Q_r$ be the set of pairs $\set{i,j}\in\smash{\binom{[\ell]}{2}}$ with $i<j$ such that $N_{i,j}$ is anticomplete to $x_{j,r}$.
Since $\calM$ is tidy, the sets $Q_0,Q_1,\ldots,Q_s$ form a partition of $\smash{\binom{[\ell]}{2}}$.
Therefore, by the choice of $\ell$, one of the conditions \ref{item:tidy-bound-a}, \ref{item:tidy-bound-b} above holds.
If \ref{item:tidy-bound-a} holds, then $I\in\calS(\calM)$, which implies $\chi(W_\ell)\leq(n-1)(c_1+c_2)\leq c'$ by \cref{lem:tidy-base}.
If \ref{item:tidy-bound-b} holds, then $(\set{x_{j,r}}_{j\in J},\set{N_{j,\ell}}_{j\in J},W_\ell)$ is an independent multicover in $G$ with $\size{J}=m$, which implies $\chi(W_\ell)\leq d\leq c'$.
\end{proof}

\begin{lemma}
\label{lem:tidy-step}
For all\/ $n,s\geq 1$, $f\colon\setN\to\setN$, and\/ $\ell,c_1,c_2,c_3\geq 0$, there is\/ $c'\geq 0$ such that the following holds.
Let\/ $G$ be a graph satisfying the following conditions:
\begin{itemize}
\item $G$ is\/ $D_n$-free;
\item $G$ is\/ $(f,s)$-clique-controlled;
\item $\chi(N(v))\leq c_1$ for every vertex\/ $v$ of\/ $G$;
\item $\chi(N^2(X))\leq c_2$ for every\/ $(s+1)$-clique\/ $X$ in\/ $G$;
\item $\chi(W(\calM))\leq c_3$ for every\/ $n$-tidy\/ $s$-clique-multicover\/ $\calM$ of length\/ $\ell+1$ in\/ $G$.
\end{itemize}
Then\/ $\chi(W(\calM'))\leq c'$ for every\/ $n$-tidy\/ $s$-clique-multicover\/ $\calM'$ of length\/ $\ell$ in\/ $G$.
\end{lemma}

\begin{proof}
Let
\vspace*{-\medskipamount}
\begin{gather*}
p_k=(k+2)2^{n-k-1}+s-1{,}\qquad q_k=k\cdot 2^{n-k-1}+1{,}\qquad\text{for }k\in[n-1]{,}\\
h=\sum_{k=1}^{n-1}\tbinom{\ell}{k}{,}\qquad p=\sum_{k=1}^{n-1}\tbinom{\ell}{k}p_k{,}\qquad q=\sum_{k=1}^{n-1}\tbinom{\ell}{k}q_k{,}\qquad c'=2^hf(pc_1+qc_2+(s+1)^\ell c_3){.}
\end{gather*}
Let $G$ be a graph as in the statement of the lemma, and let
\[\calM'=(W_0,\;X_1,N_{1,1},W_1,\;X_2,N_{1,2},N_{2,2},W_2,\;\ldots,\;X_\ell,N_{1,\ell},\ldots,N_{\ell,\ell},W_\ell)\]
be an $n$-tidy $s$-clique-multicover in $G$.
We aim to prove that $\chi(W_\ell)\leq c'$.

A subset of $V(G)$ is \emph{$(\bar p,\bar q)$-small} if it is contained in $\bigcup_{v\in U}N(v)\cup\bigcup_{X\in\calX}N^2(X)$ for some set $U\subseteq V(G)$ with $\size{U}\leq\bar p$ and some set $\calX$ of $(s+1)$-cliques in $G$ with $\size{\calX}\leq\bar q$.
Observe that every $(\bar p,\bar q)$-small set $S\subseteq V(G)$ satisfies $\chi(S)\leq\bar pc_1+\bar qc_2$.

Let $\calH=\set{J\in\calS(\calM')\colon 1\leq\size{J}\leq n-1}$.
Thus $\size{\calH}\leq h$.
For each $v\in W_\ell$, let $\Phi(v)$ comprise the sets $J\in\calH$ such that for every $(p_{\size{J}},q_{\size{J}})$-small set $S\subseteq W_\ell$, there is a $(\set{N_{i,\ell}\cap N(v)}_{i\in J},\:W_\ell\setminus(N[v]\cup S))$-embedding of $D_{n-\size{J}}$ in $G$.
For each $\Phi\subseteq\calH$, let $W_\Phi=\set{v\in W_\ell\colon\Phi(v)=\Phi}$.
The sets $W_\Phi$ partition $W_\ell$.
We prove that for every $\Phi\subseteq\calH$ and every $s$-clique $X$ in $W_\Phi$, we have
\[\chi(N_{G[W_\Phi]}^2(X))\leq pc_1+qc_2+(s+1)^\ell c_3{.}\]
Then, since $G$ is $(f,s)$-clique-controlled, we can conclude that $\chi(W_\Phi)\leq f(pc_1+qc_2+(s+1)^\ell c_3)$ for every $\Phi\subseteq\calH$, which implies
\[\chi(W_\ell)\leq\sum_{\Phi\subseteq\calH}\chi(W_\Phi)\leq 2^hf(pc_1+qc_2+(s+1)^\ell c_3)=c'{.}\]

Fix some $\Phi\subseteq\calH$ and some $s$-clique $X_{\ell+1}$ in $W_\Phi$.
Fix $x^*\in X_{\ell+1}$.
For every $J\in\calH\setminus\Phi$, since $\Phi(x^*)=\Phi$, there is a $(p_{\size{J}},q_{\size{J}})$-small set $S_J\subseteq W_\ell$ such that $G$ contains no $(\set{N_{i,\ell}\cap N(x^*)}_{i\in J},\:W_\ell\setminus(N[x^*]\cup S_J))$-embedding of $D_{n-\size{J}}$.
Let $S=\bigcup_{J\in\calH\setminus\Phi}S_J$.
It follows that $S$ is $(p,q)$-small, which implies $\chi(S)\leq pc_1+qc_2$.
Let $N_{\ell+1,\ell+1}=\smash[b]{N_{G[W_\Phi]}^1}(X_{\ell+1})$ and $W'=\smash[b]{N_{G[W_\Phi]}^2}(X_{\ell+1})\setminus S$, both of which are subsets of $W_\ell$.
It follows that $G$ contains no $(\set{N_{i,\ell}\cap N(x^*)}_{i\in J},\:W'\setminus N[x^*])$-embedding of $D_{n-\size{J}}$ for any $J\in\calH\setminus\Phi$, and
\[\chi(N_{G[W_\Phi]}^2(X_{\ell+1}))\leq\chi(S)+\chi(W')\leq pc_1+qc_2+\chi(W'){.}\]

Let $Y=X_{\ell+1}\cup\set{y}$ where $y$ is some symbol different from all members of $X_{\ell+1}$.
For each vertex $v\in W'$, let $\Psi(v)$ be a function $[\ell]\to Y$ such that the following holds for every $i\in[\ell]$:
\begin{itemize}
\item if $X_{\ell+1}$ is complete to $N_{i,\ell}\cap N(v)$, then $\Psi(v)(i)=y$,
\item otherwise, $\Psi(v)(i)=x$ for some $x\in X_{\ell+1}$ that is not complete to $N_{i,\ell}\cap N(v)$.
\end{itemize}
For each function $\Psi\colon[\ell]\to Y$, let $W'_\Psi=\set{v\in W'\colon\Psi(v)=\Psi}$.
There are $(s+1)^\ell$ functions $\Psi$, and the sets $W'_\Psi$ partition $W'$.
We prove that $\chi(W'_\Psi)\leq c_3$ for every function $\Psi$.
Then, we can conclude that
\[\chi(N_{G[W_\Phi]}^2(X_{\ell+1}))\leq pc_1+qc_2+\chi(W')\leq pc_1+qc_2+\sum_{\Psi:[\ell]\to Y}\chi(W'_\Psi)\leq pc_1+qc_2+(s+1)^\ell c_3{.}\]

Fix a function $\Psi$, and let $W_{\ell+1}=W'_\Psi$.
We aim to prove that $\chi(W_{\ell+1})\leq c_3$.
For every $i\in[\ell]$, let $N_{i,\ell+1}$ be defined as follows:
\begin{itemize}
\item if $\Psi(i)=y$, then $N_{i,\ell+1}$ is the set of vertices in $N_{i,\ell}$ that are complete to $X_{\ell+1}$;
\item if $\Psi(i)\in X_{\ell+1}$, then $N_{i,\ell+1}$ is the set of vertices in $N_{i,\ell}$ that are anticomplete to $\Psi(i)$.
\end{itemize}
The above implies the following, for every $i\in[\ell]$:
\begin{itemize}
\item if $\Psi(i)=y$, then $\emptyset\neq N_{i,\ell}\cap N(v)\subseteq N_{i,\ell+1}$ for every $v\in W_{\ell+1}$, so $N_{i,\ell+1}$ covers $W_{\ell+1}$ while $N_{i,\ell}\setminus N_{i,{\ell+1}}$ is anticomplete to $W_{\ell+1}$;
\item if $\Psi(i)\in X_{\ell+1}$, then $N_{i,\ell+1}\cap N(v)\neq\emptyset$ for every $v\in W_{\ell+1}$, so $N_{i,\ell+1}$ covers $W_{\ell+1}$.
\end{itemize}
Consequently, the following is an $s$-clique-multicover in $G$:
\[\calM=(W_0,\;X_1,N_{1,1},W_1,\;X_2,N_{1,2},N_{2,2},W_2,\;\ldots,\;X_{\ell+1},N_{1,\ell+1},\ldots,N_{\ell+1,\ell+1},W_{\ell+1}){.}\]
Moreover, every pair of indices $i$ and $j$ with $1\leq i<j\leq\ell+1$ is independent or skew in $\calM$, which follows from the above for $j=\ell+1$ and from the tidiness of $\calM'$ for $j\leq\ell$.
We conclude that $\calM$ is a tidy $s$-clique-multicover in $G$.
We prove that $\calM$ is $n$-tidy, which then implies $\chi(W_{\ell+1})\leq c_3$, as required.

To prove that $\calM$ is $n$-tidy, we need to verify that for every set $I\in\calS(\calM)$ with $1\leq\size{I}\leq n$, there is no $(\set{N_{i,\max I}}_{i\in I},W_{\max I})$-embedding of $D_{n-\size{I}+1}$ in $G$.
For sets $I$ with $\max I\leq\ell$, we have $I\in\calS(\calM')$, and the property follows from the $n$-tidiness of $\calM'$.
Now, let $I\in\calS(\calM)$ where $1\leq\size{I}\leq n$ and $\max I=\ell+1$.
If $\size{I}=1$, then the property follows from the assumption that $G$ is $D_n$-free.
Suppose towards a contradiction that $\size{I}\geq 2$ and there is a $(\set{N_{i,\ell+1}}_{i\in I},W_{\ell+1})$-embedding $
\set{\phi_i}_{i\in I}$ of $D_{n-\size{I}+1}$ in $G$.
Let $J=I\setminus\set{\ell+1}$ and $k=n-\size{I}+1=n-\size{J}$.
Let $A_i=\phi_i(V_A(D_k))$ and $B=\phi_i(V_N(D_k))$ for $i\in I$, where the latter does not depend on $i$ by the definition of $(\set{N_{i,\ell+1}}_{i\in I},W_{\ell+1})$-embedding.

Recall that $(D_k)_{k\geq 1}$ and $(a^*_k)_{k\geq 1}$ are sequences claimed in \cref{lem:alt-burling}.
Let $D'_k$ denote the isomorphic copy of $D_k$ and $b^*_k$ denote the vertex in $V(D_{k+1})$ corresponding to $b^*$ used in the definition of $D_{k+1}$ as $\Dup(D_k,a^*_k)$.
Let $v^*=\phi_{\ell+1}(a^*_k)$.
Since $I\in\calS(\calM)$ and $A_i=\phi_i(V_A(D_k))\subseteq N_{i,\ell+1}$ for all $i\in I$, the sets $X_{\ell+1}\cup\set{v^*}$ and $X_{\max J}\cup\set{u}$ for all $u\in\bigcup_{i\in J}A_i$ are $(s+1)$-cliques.
Let $S$ be the union of the following sets:
\begin{itemize}
\item $N(v)$ for all $v\in\bigcup_{i\in J}A_i\cup B\cup X_{\ell+1}$,
\item $N^2(X_{\ell+1}\cup\set{v^*})$,
\item $N^2(X_{\max J}\cup\set{u})$ for all $u\in\bigcup_{i\in J}A_i$.
\end{itemize}
Since $\size{V_A(D_k)}=2^{k-1}$ and $\size{V_N(D_k)}=2\cdot 2^{k-1}-1$, the set $S$ is $(p_{\size{J}},q_{\size{J}})$-small.

Recall that we have fixed a vertex $x^*\in X_{\ell+1}$.
Since $N_{i,\ell+1}\subseteq N(x^*)$ for all $i\in I$ and $W_{\ell+1}\subseteq W'\setminus N[x^*]$, the family $\set{\phi_i}_{i\in J}$ is a $(\set{N_{i,\ell}\cap N(x^*)}_{i\in J},\:W'\setminus N[x^*])$-embedding of $D_k$ in $G$.
We have guaranteed that such an embedding cannot exist if $J\in\calH\setminus\Phi$, so $J\in\Phi$.
Since $v^*\in N_{\ell+1,\ell+1}\subseteq W_\Phi$, we have $J\in\Phi=\Phi(v^*)$, which implies (by definition and the fact that $D_k$ and $D'_k$ are isomorphic) that there is a $(\set{N_{i,\ell}\cap N(v^*)}_{i\in J},\:W_\ell\setminus(N[v^*]\cup S))$-embedding $\set{\phi'_i}_{i\in J}$ of $D'_k$ in $G$.
Let $A'_i=\phi'_i(V_A(D'_k))$ and $B'=\phi'_i(V_N(D'_k))$ for $i\in J$, where the latter again does not depend on $i$.

For every $i\in J$, let $\psi_i=\phi_i\cup\phi'_i\cup\set{b^*_k\mapsto v^*}$.
We claim that $\set{\psi_i}_{i\in J}$ is a $(\set{N_{i,\ell}}_{i\in J},W_\ell)$-embedding of $D_{k+1}$ in $G$.
The definitions imply that $\psi_i(V_A(D_{k+1}))\subseteq N_{i,\ell}$, $\psi_i(V_N(D_{k+1}))\subseteq W_\ell$, and all $\psi_i$ with $i\in J$ agree on $V_N(D_{k+1})$, so it remains to verify the following conditions:
\begin{itemize}
\item $\psi_i$ is an embedding of $D_{k+1}$ in $G[N_{i,\ell}\cup W_\ell]$ for all $i\in J$;
\item $\bigcup_{i\in J}(A_i\cup A'_i)$ is an independent set in $G$.
\end{itemize}
We already know that $\bigcup_{i\in J}A_i$ and $\bigcup_{i\in J}A'_i$ are independent sets.
We also know that $\bigcup_{i\in J}A'_i\subseteq N(v^*)$, $\bigl(\bigcup_{i\in I}A_i\cup B'\bigr)\cap N(v^*)=\emptyset$, and $B\cap N(v^*)=\phi_i(N_{D_k}(a^*_k))$, as required.
Therefore, to verify the two conditions above, it suffices to prove that $\bigcup_{i\in J}A'_i\cup B'$ is anticomplete to $\bigcup_{i\in J}A_i\cup B$.

The set $B'$ is anticomplete to $\bigcup_{i\in J}A_i\cup B$ because all neighbors of the vertices in the latter set belong to $S$ while $B'\cap S=\emptyset$.
Now, consider a vertex $v\in A'_i$ with $i\in J$.
Since the vertex $a\in V_A(D'_k)$ with $v=\phi'_i(a)$ has a neighbor $b\in V_N(D'_k)$, the vertex $v$ has a neighbor $\phi'_i(b)\in B'$.
We have $v\in N_{i,\ell}\cap N(v^*)$.
If $v\in N_{i,\ell+1}$, then the fact that $N_{i,\ell+1}$ is complete to $X_{\ell+1}$ implies $v\in N^1(X_{\ell+1}\cup\set{v^*})$, so $N(v)\cap W_\ell\subseteq N[v^*]\cup\bigcup_{x\in X_{\ell+1}}N(x)\cup N^2(X_{\ell+1}\cup\set{v^*})\subseteq N[v^*]\cup S$ while $B'\subseteq W_\ell\setminus(N[v^*]\cup S)$, contradicting the fact that $v$ has a neighbor in $B'$.
Thus $v\in N_{i,\ell}\setminus N_{i,\ell+1}$, and since the latter set is anticomplete to $W_{\ell+1}$ while $B\subseteq W_{\ell+1}$, we infer that $v$ is anticomplete to $B$.
If $v\in N_{i,\ell}\cap N(u)$ for some $u\in A_j$ with $j\in J$, then the fact that $N_{j,\ell}\subseteq N_{j,\max J}$ and the latter set is complete to $X_{\max J}$ implies $v\in N^1(X_{\max J}\cup\set{u})$, so $N(v)\cap W_\ell\subseteq N(u)\cup N^2(X_{\max J}\cup\set{u})\subseteq S$ (as $W_\ell$ is anticomplete to $X_{\max J}$), again contradicting the fact that $v$ has a neighbor in $B'$.
This shows that $v$ is anticomplete to $\bigcup_{i\in J}A_i$.

We have shown that $\set{\psi_i}_{i\in J}$ is a $(\set{N_{i,\ell}}_{i\in J},W_\ell)$-embedding of $D_{k+1}$ in $G$.
Since $J\in\calS(\calM')$, this contradicts the assumption that $\calM'$ is $n$-tidy.
Thanks to this contradiction, we conclude that $\calM$ is an $n$-tidy $s$-clique-multicover in $G$, which implies $\chi(W_{\ell+1})\leq c_3$.
\end{proof}

\begin{lemma}
\label{lem:final-bound}
For all\/ $n,s\geq 1$, $f\colon\setN\to\setN$, and\/ $\ell,c_1,c_2,c_3\geq 0$, there is\/ $c'\geq 0$ such that the following holds.
Let\/ $G$ be a graph satisfying the following conditions:
\begin{itemize}
\item $G$ is\/ $D_n$-free;
\item $G$ is\/ $(f,s)$-clique-controlled;
\item $\chi(N(v))\leq c_1$ for every vertex\/ $v$ of\/ $G$;
\item $\chi(N^2(X))\leq c_2$ for every\/ $(s+1)$-clique\/ $X$ in\/ $G$;
\item $\chi(W(\calM))\leq c_3$ for every\/ $n$-tidy\/ $s$-clique-multicover\/ $\calM$ of length\/ $\ell$ in\/ $G$.
\end{itemize}
Then $\chi(G)\leq c'$.
\end{lemma}

\begin{proof}
For any $\ell,c_3\geq 0$, let $f_\ell(c_3)$ be the constant $c'$ claimed by \cref{lem:tidy-step} for $s,n,f,\ell,c_1,c_2,c_3$.
We show that the lemma holds with $c'=(f_1\circ\cdots\circ f_\ell)(c_3)$.
To this end, we prove that for every $k\in\set{0,1,\ldots,\ell}$, every $n$-tidy $s$-clique-multicover $\calM_k$ of length $k$ in $G$ satisfies $\chi(W(\calM_k))\leq(f_{k+1}\circ\cdots\circ f_\ell)(c_3)$.
Then, for $k=0$ and the trivial $n$-tidy $s$-clique-multicover $\calM_0=(V(G))$, we can conclude that
\[\chi(G)=\chi(W(\calM_0))\leq(f_1\circ\cdots\circ f_\ell)(c_3)=c'{.}\]

The proof of the claim goes by downward induction on $k$.
The base case of $k=\ell$ follows by the last assumption of the lemma.
For the induction step with $0\leq k<\ell$, \cref{lem:tidy-step} applied to an $n$-tidy $s$-clique-multicover $\calM_k$ of length $k$ with $(f_{k+2}\circ\cdots\circ f_\ell)(c_3)$ playing the role of $c_3$ gives $\chi(W(\calM_k))\leq(f_{k+1}\circ\cdots\circ f_\ell)(c_3)$, as required.
\end{proof}

We are ready to complete the proof of \cref{lem:main}, which we restate for convenience.

\lemmain*

\begin{proof}
By \cref{lem:burling,lem:alt-burling}, there exists $n$ such that every graph $G$ with $\beta(G)\leq r$ is $D_n$-free.
We proceed to define functions $f_1,\ldots,f_k\colon\setN\to\setN$ such that every graph $G$ satisfying the assumptions of the lemma is $(f_s,s)$-clique-controlled for every $s\in[k]$.
This is done by induction on $s$.
For the base case of $s=1$, let $f_1(c)=f(c+c_1+1)$ for all $c\geq 0$.
The above implies that for any $v\in V(G)$ and $c\geq 0$, if $\chi(N^2(\set{v}))\leq c$, then
\[\chi(N^2[v])=\chi(N^2(\set{v})\cup N[v])\leq\chi(N^2(\set{v}))+\chi(N[v])\leq c+c_1+1{.}\]
The same argument can be applied with any induced subgraph of $G$ playing the role of $G$, which shows that $G$ is $(f_1,1)$-clique-controlled.
For the induction step, suppose that $s\in[k-1]$ and we already have $f_s$, and for all $c_2\geq 0$, let $f_{s+1}(c_2)$ be the constant $c'$ claimed by \cref{lem:final-bound} with $f_s$ playing the role of $f$ and with the constants $\ell,c'\geq 0$ claimed by \cref{lem:tidy-bound} playing the role of $\ell,c_3$.
By the latter, $\chi(W(\calM))\leq c_3$ for every $n$-tidy $s$-clique-multicover $\calM$ of length $\ell$ in $G$, and therefore, since $G$ is $(f_s,s)$-clique-controlled, \cref{lem:final-bound} yields $\chi(G)\leq f_{s+1}(c_2)$.
The same argument can be applied with any induced subgraph of $G$ playing the role of $G$, which shows that $G$ is $(f_{s+1},s+1)$-clique-controlled.

We conclude that $G$ is $(f_k,k)$-clique-controlled, which implies that $\chi(G)\leq f_k(0)$, because $N^2(X)=\emptyset$ for every $k$-clique $X$ in $G$.
\end{proof}

\section{Region intersection graphs}
\label{sec:region}

Recall that a \emph{region intersection model} of a graph $G$ over a graph $H$ is a mapping $\mu$ from $V(G)$ to connected sets in $H$ such that $\mu(u)\cap\mu(v)\neq\emptyset$ exactly when $uv$ is an edge of $G$.
This mapping extends to connected sets $S\subseteq V(G)$ by setting $\mu(S)=\bigcup_{v\in S}\mu(v)$, which is a connected set in $H$.

A \emph{$t$-model} in a graph $H$ is a family of $t$ connected sets in $H$ that are pairwise disjoint and adjacent in $H$.
Let $\mu$ be a region intersection model of a graph $G$ over $H$, let $S\subseteq V(G)$, and let $\gamma\geq 0$.
A~\emph{$(\mu,S,\gamma,t)$-model} is a family $\set{(x_j,T_j)}_{j\in[t]}$ such that the following holds:
\begin{enumeratei}
\item\label{item:model-1} $\set{T_j}_{j\in[t]}$ is a $t$-model in $H$;
\item\label{item:model-2} $x_j\in S$ for every $j\in[t]$;
\item\label{item:model-3} $\mu(N^\gamma[x_j])\subseteq T_j$ for every $j\in[t]$.
\end{enumeratei}

Let $G$ be a graph and $\delta\geq 0$.
A \emph{$\delta$-component} of a set $U\subseteq V(G)$ in $G$ is an inclusion-minimal non-empty set $C\subseteq U$ such that $d(u,v)>\delta$ for all $u\in C$ and $v\in U\setminus C$.
For $\rho\geq 0$, a \emph{$(\delta,\rho)$-partition} of a set $S\subseteq V(G)$ in $G$ is
a partition $\calU$ of $S$ such that for every $U\in\calU$, every $\delta$-component of $U$ has radius at most $\rho$.
The definition of asymptotic dimension can be rephrased as follows: the \emph{asymptotic dimension} of a graph class $\calG$ is the minimum $d$ such that for every $\delta\geq 0$, there is $\rho\geq 0$ such that every graph $G\in\calG$ admits a $(\delta,\rho)$-partition of $V(G)$ of size at most $d+1$.

\begin{lemma}
\label{lem:region-step}
Let\/ $G$ and\/ $H$ be graphs, let\/ $\mu$ be a region intersection model of\/ $G$ over\/ $H$, and let\/ $S\subseteq V(G)$.
Let\/ $t\geq 2$ and\/ $\gamma\geq\delta\geq 1$.
If there is no\/ $(\mu,S,\gamma,t)$-model, then\/ $S$ has a\/ $(\delta,2\gamma+1+2(t-2)\delta)$-partition in\/ $G$ of size at most\/ $2^{t-2}$.
Moreover, such a partition can be computed in\/ $C(t,\ceil{\frac{\gamma}{\delta}})n$ time given\/ $G$, $H$, $\mu$, and\/ $S$ in the input, where\/ $n$ denotes the size of the input.
\end{lemma}

\begin{proof}
The proof goes by induction on $t$.
Let $t\geq 2$, and if $t\geq 3$, then suppose the lemma holds for $t-1$.
Let $\rho=2\gamma+1+2(t-2)\delta$.
If $S$ has more than one $\delta$-component in $G$, then a suitable $(\delta,\rho)$-partition can be computed separately for every $\delta$-component of $S$ in $G$, so assume $S$ is itself a $\delta$-component in $G$.
Fix some $x^*\in S$.
For every $i\geq 0$, let $V_i=\set{v\in V(G)\colon d_G(x^*,v)=i}$.
For all $j\geq i\geq 0$, let $V_{i,j}=V_i\cup V_{i+1}\cup\cdots\cup V_j$, and let $\mu_{i,j}$ be the restriction of $\mu$ to the set $V_{i,j}$.
For all $i\geq 1$, let $\calC_i$ be the set of connected components of all graphs $H[\mu(v)\setminus\mu(V_{0,i-1})]$ with $v\in V_i$, let $V'_i$ be a set of size $\size{\calC_i}$ disjoint from $V(G)$,
and let $\mu'_i\colon V'_i\to\calC_i$ be an arbitrary bijection.
For all $j>i\geq 1$, let $V'_{i,j}=V'_i\cup V_{i+1,j}$, let $\mu'_{i,j}=\mu'_i\cup\mu_{i+1,j}$, and let $G'_{i,j}$ be the graph with vertex set $V'_{i,j}$ for which $\mu'_{i,j}$ is a region intersection model over $H$.
Note that $G[V_{i+1,j}]$ is an induced subgraph of $G'_{i,j}$.

Let $i\geq 2\gamma+1$.
We prove that every $(\mu'_{i-\gamma,i+\gamma+\delta},\:S\cap V_{i+1,i+\delta},\:\gamma+\delta,\:t-1)$-model $\set{(x_j,T_j)}_{j\in[t-1]}$ extends to a $(\mu,S,\gamma,t)$-model $\set{(x_j,T_j)}_{j\in[t]}$ where $x_t=x^*$ and $T_t=\mu(V_{0,i-\gamma-1})$.
To this end, we need to verify conditions \ref{item:model-1}--\ref{item:model-3} from the definition of $(\mu,S,\gamma,t)$-model.
Condition \ref{item:model-2} follows directly from the same condition on $\set{(x_j,T_j)}_{j\in[t-1]}$ and from the fact that $x^*\in S$.

Condition \ref{item:model-3} for $j=t$ follows from the definition of $T_t$, and for $j\in[t-1]$, since $G[V_{i-\gamma+1,i+\gamma+\delta}]$ is an induced subgraph of $G'_{i-\gamma,i+\gamma+\delta}$ and $x_j\in S\cap V_{i+1,i+\delta}$, we have $N_G^\gamma[x_j]\subseteq V_{i-\gamma+1,i+\gamma+\delta}$, which implies $\mu(N_G^\gamma[x_j])=\mu'_{i-\gamma,i+\gamma+\delta}(N_{G'_{i-\gamma,i+\gamma+\delta}}^\gamma[x_j])\subseteq T_j$.
For the proof of \ref{item:model-1}, observe that the sets in $\set{T_j}_{j\in[t-1]}$ are pairwise disjoint and adjacent by assumption.
They are also disjoint from $T_t$, because $\mu(V_{0,i-\gamma-1})\cap\mu'_{i-\gamma}(V'_{i-\gamma})=\emptyset$ by definition and $\mu(V_{0,i-\gamma-1})\cap\mu(V_{i-\gamma+1,i+\gamma+\delta})=\emptyset$ as $G$ has no edge between $V_{0,i-\gamma-1}$ and $V_{i-\gamma+1,i+\gamma+\delta}$.

It remains to show that for every $j\in[t-1]$, $T_j$ is adjacent to $T_t$.
Let $j\in[t-1]$, and let $\ell\in\set{i+1,\ldots,i+\delta}$ be such that $x_j\in S\cap V_\ell$.
Let $v_0v_1\cdots v_\ell$ be a shortest path in $G$ with $v_0=x^*$ and $v_\ell=x_j$.
Since $\mu(v_{i-\gamma})\cap\mu(v_{i-\gamma+1})\neq\emptyset$ while $\mu(V_{0,i-\gamma-1})\cap\mu(v_{i-\gamma+1})=\emptyset$, there is a connected component $C\in\calC_{i-\gamma}$ of $\mu(v_{i-\gamma})\setminus\mu(V_{0,i-\gamma-1})$ with $C\cap\mu(v_{i-\gamma+1})\neq\emptyset$.
Let $v'\in V'_{i-\gamma}$ be such that $\mu'_{i-\gamma}(v')=C$.
It follows that $\mu'_{i-\gamma,i+\gamma+\delta}(v')\cap\mu'_{i-\gamma,i+\gamma+\delta}(v_{i-\gamma+1})=C\cap\mu(v_{i-\gamma+1})\neq\emptyset$, so $v'v_{i-\gamma+1}\cdots v_\ell$ is a shortest path in $G'_{i-\gamma,i+\gamma+\delta}$.
Consequently, $v'$ is at distance $\ell-(i-\gamma)\leq\gamma+\delta$ from $x_j=v_\ell$ in $G'_{i-\gamma,i+\gamma+\delta}$.
The set $\mu'_{i-\gamma}(v')$, as a connected component of $\mu(v_{i-\gamma})\setminus\mu(V_{0,i-\gamma-1})$, is adjacent to $\mu(V_{0,i-\gamma-1})=T_t$.
Since $\mu'_{i-\gamma}(v')\subseteq\mu'_{i-\gamma,i+\gamma+\delta}(\smash[t]{N_{G'_{i-\gamma,i+\gamma+\delta}}^{\gamma+\delta}}[x_j])\subseteq T_j$, the set $T_j$ is indeed adjacent to $T_t$.
We conclude that $\set{(x_j,T_j)}_{j\in[t]}$ is indeed a $(\mu,S,\gamma,t)$-model.

Suppose $t=2$.
If there were $y^*\in S\cap V_{2\gamma+2,2\gamma+1+\delta}$, then $\set{(y^*,\mu'_{\gamma+1,3\gamma+1+\delta}(\smash[t]{N_{G'_{\gamma+1,3\gamma+1+\delta}}^{\gamma+\delta}}[y^*]))}$ would be a (trivial) $(\mu'_{\gamma+1,3\gamma+1+\delta},\:S\cap V_{2\gamma+2,2\gamma+1+\delta},\:\gamma+\delta,\:1)$-model, which would extend to a $(\mu,S,\gamma,2)$-model as shown above, contradicting the assumption that there is no such model.
Thus $S\cap V_{2\gamma+2,2\gamma+1+\delta}=\emptyset$.
Since $S$ is a $\delta$-component in $G$, we have $S\subseteq V_{0,2\gamma+1}$.
Since the set $V_{0,2\gamma+1}$ has radius at most $2\gamma+1=\rho$ in $G$, we conclude that $\set{S}$ is a $(\delta,\rho)$-partition of $S$ in $G$ of size $1$.

Now, suppose $t\geq 3$.
Let $I=\set{2\gamma+1+q\delta\colon q\in\setN}$.
Since there is no $(\mu,S,\gamma,t)$-model, the above implies that there is no $(\mu'_{i-\gamma,i+\gamma+\delta},\:S\cap V_{i+1,i+\delta},\:\gamma+\delta,\:t-1)$-model for any $i\in I$.
Therefore, for every $i\in I$, the induction hypothesis provides a $(\delta,\rho)$-partition $\set{U_{i,k}}_{k\in[2^{t-3}]}$ of $S\cap V_{i+1,i+\delta}$ in $G'_{i-\gamma,i+\gamma+\delta}$.
For $r\in\set{0,1}$, let $I_r=\{2\gamma+1+q\delta\colon q\in\setN$ and $q\equiv r\pmod{2}\}$.
For $r\in\set{0,1}$ and $k\in[2^{t-3}]$, set
\[U_{r\cdot 2^{t-3}+k} = \begin{cases}
\bigcup_{i\in I_r}U_{i,k} &\text{ if } (r,k)\neq(1,1){,}\\
(S\cap V_{0,2\gamma+1}) \cup \bigcup_{i\in I_r}U_{i,k} &\text{ if } (r,k)=(1,1){.}
\end{cases}\]

We prove that $\set{U_k}_{k\in[2^{t-2}]}$ is a $(\delta,\rho)$-partition in $G$.
Since the sets $S\cap V_{0,2\gamma+1}$ and $S\cap V_{i+1,i+\delta}$ with $i\in I$ are pairwise disjoint and $S$ is their union, $\set{U_k}_{k\in[2^{t-2}]}$ is a partition of $S$.
Let $r\in\set{0,1}$ and $k\in[2^{t-3}]$.
If $u,v\in U_{r\cdot 2^{t-3}+k}$ while $u$ and $v$ do not belong to the same set of the form $U_{i,k}$ with $i\in I_r$ or $S\cap V_{0,2\gamma+1}$ when $(r,k)=(1,1)$, then $u$ and $v$ are separated in $G$ by the set $V_{i+1,i+\delta}$ for some $i\in I_{1-r}$, so $d_G(u,v)\geq\delta+1$.
This shows that every $\delta$-component $C$ of $U_{r\cdot 2^{t-3}+k}$ is contained either in $U_{i,k}$ for some $i\in I_r$ or in $S\cap V_{0,2\gamma+1}$ when $(r,k)=(1,1)$.
In the former case, since $U_{i,k}\subseteq V_{i+1,i+\delta}$ and $V_{i-\delta+1,i+2\delta}\subseteq V_{i-\gamma+1,i+\gamma+\delta} \subseteq V(G'_{i-\gamma,i+\gamma+\delta})$ (as $\gamma\geq\delta$), $C$ is a $\delta$-component of $U_{i,k}$ in $G'_{i-\gamma,i+\gamma+\delta}$, so it has radius at most $\rho$ by the induction hypothesis.
In the latter case, $C$ has radius at most $2\gamma+1\leq\rho$.
We conclude that $\set{U_k}_{k\in[2^{t-2}]}$ is indeed a $(\delta,\rho)$-partition in $G$.

For the complexity analysis, since the mapping $\mu$ stores all necessary information about $G$, we can assume that the input contains just $\mu$ and $S$, so the size of an input instance can be measured as $n=\size{\mu}=\sum_{v\in V(G)}\size{\mu(v)}$, because $\size{S}\leq\size{\mu}$.
We prove that the running time is bounded by $C(t,\lambda)n$, where $\lambda=\ceil{\frac{\gamma}{\delta}}$.
Computing $\set{U_k}_{k\in[2^{t-2}]}$ involves recursive calls on $\mu'_{i-\gamma,i+\gamma+\delta}$ for all $i\in I$ with $V_{i+1}\neq\emptyset$.
We have
\[\sum_{i\in I}\size{\mu'_{i-\gamma,i+\gamma+\delta}}\leq\sum_{i\in I}\bigl(\size{\mu'_{i-\gamma}}+\size{\mu_{i-\gamma+1,i+\gamma+\delta}}\bigr)\leq\size{\mu}+(2\lambda+1)\size{\mu}=(2\lambda+2)\size{\mu}{.}\]
So the total time spent on the recursive calls is at most $(2\lambda+2)C(t-1,\lambda+1)n$ when $t\geq 3$.
Since the additional work requires at most $C'n$ time, we get
\[C(2,\lambda)=C'{,}\qquad C(t,\lambda)=(2\lambda+2)C(t-1,\lambda+1)+C'\quad\text{for }t\geq 3{.}\qedhere\]
\end{proof}

Now, we are ready to prove \cref{thm:region-asdim}, which we state with an explicit control function.

\begin{theorem}
\label{thm:region-asdim-explicit}
Let\/ $t\geq 2$, and let\/ $H$ be a graph excluding\/ $K_t$ as a minor.
For every region intersection graph\/ $G$ over\/ $H$ and for every\/ $\delta\geq 0$, the set\/ $V(G)$ admits a\/ $(2^{t-2},\delta)$-disjoint partition into sets of radius at most\/ $2(t-1)\delta+1$ in\/ $G$.
Consequently, the class of region intersection graphs over\/ $H$ has asymptotic dimension at most\/ $2^{t-2}-1$ with\/ $(2^{t-2}-1)$-dimensional control function\/ $\delta\mapsto 4(t-1)\delta+2$.
\end{theorem}

\begin{proof}
This is trivial for $\delta=0$, as $\set{V(G)}$ is a $(0,1)$-partition of $V(G)$ in $G$.
Let $\delta\geq 1$, and let $\mu$ be a region intersection model of a graph $G$ over $H$.
Since $K_t$ is not a minor of $H$, there is no $(\mu,V(G),\delta,t)$-model.
Therefore, by \cref{lem:region-step}, there is a $(\delta,2(t-1)\delta+1)$-partition of $V(G)$ in $G$ of size at most $2^{t-2}$.
The $\delta$-components of the sets in this partition forms a $(2^{t-2}-1,\delta)$-disjoint partition of $V(G)$ into sets of radius at most $2(t-1)\delta+1$ in $G$, as claimed.
The second statement follows from the first one, because the sets in the partition have diameter at most $4(t-1)\delta+2$.
\end{proof}

By the last statement of \cref{lem:region-step}, given $G$, $H$, $\mu$, and $\delta$ in the input, a $(\delta,2(t-1)\delta+1)$-partition of $V(G)$ in $G$ of size at most $2^{t-2}$ claimed by \cref{thm:region-asdim-explicit} can be computed in $C(t,1)n$ time.

\section{Finite asymptotic dimension}
\label{sec:asdim}

This section is devoted to the proofs of the results concerning graph classes with finite asymptotic dimension, that is, \cref{prop:asdim-subdivision,thm:asdim-2-controlled}, which we restate.

\propasdimsubdivision*

\begin{proof}
Let $\calG$ be a hereditary graph class with finite asymptotic dimension $d$, and let $r\geq 1$.
Let $f\colon\setN\to\setN$ be a $d$-dimensional control function for the graphs in $\calG$.
By a well-known result of Erdős~\cite{Erd59}, there is a graph $F$ with chromatic number greater than $2d+2$ and with no cycles of length less than $2f(r+1)+2$.
Suppose towards a contradiction that $F'\in\calG$ for some $(\leq r)$-subdivision $F'$ of $F$.
The set $V(F')$ comprises the original vertices of $F$ and the new subdividing vertices.
Let $\calU=\bigcup_{i=1}^{d+1}\calU_i$ be a $(d+1,r+1)$-disjoint partition of $V(F')$ into sets of diameter at most $f(r+1)$ in $F'$, where $\calU_i$ is $(r+1)$-disjoint for every $i\in[d+1]$.

Consider a set $S'\in\calU$.
Let $S=S'\cap V(F)$, and suppose that $S$ is non-empty.
Since the diameter of $S'$ in $F'$ is at most $f(r+1)$, so is the diameter of $S$ in $F$.
This implies that $F[S]$ is a tree, because every cycle in $F$ contains a pair of vertices at distance at least $f(r+1)$ in $F$.

Now, let $i\in[d+1]$, and let $U_i=\bigl(\bigcup\calU_i\bigr)\cap V(F)$.
For any distinct $S'_1,S'_2\in\calU_i$, there is no edge $v_1v_2$ in $F$ with $v_1\in S'_1\cap V(F)$ and $v_2\in S'_2\cap V(F)$, because such an edge would imply $d_{F'}(v_1,v_2)\leq r+1$, contradicting the assumption that $\calU_i$ is $(r+1)$-disjoint.
This implies that $F[U_i]$ is a forest.
Consequently, $\chi(F[U_i])\leq 2$.
We conclude that
\[\chi(F)\leq\sum_{i=1}^{d+1}\chi(F[U_i])\leq 2d+2{,}\]
which contradicts the choice of $F$ as a graph with $\chi(F)>2d+2$.

We have shown that no $(\leq r)$-subdivision of $F$ belongs to $\calG$.
Since $\calG$ is hereditary, we conclude that no graph in $\calG$ contains an induced $(\leq r)$-subdivision of $F$.
\end{proof}

\thmasdimcontrol*

\begin{proof}
Let $\calG$ be a hereditary graph class with finite asymptotic dimension $d$.
Let $f\colon\setN\to\setN$ be a $d$-dimensional control function for the graphs in $\calG$, and let $\rho=f(1)$.
It follows that the class $\calG$ is $\rho$-controlled.
Indeed, if $G\in\calG$ and $\calU=\bigcup_{i=1}^{d+1}\calU_i$ is a $(d+1,1)$-disjoint partition of $V(G)$ into sets of diameter at most $\rho$ in $G$, where $\calU_i$ is $1$-disjoint for every $i\in[d+1]$, then
\[\chi(G)\leq\sum_{i=1}^{d+1}\chi\bigl({\textstyle\bigcup\calU_i}\bigr)=\sum_{i=1}^{d+1}\max_{S\in\calU_i}\chi(S)\leq(d+1)\max_{v\in V(G)}\chi(N_G^\rho[v]){.}\]

Let $F$ be a graph claimed by \cref{prop:asdim-subdivision} for $\calG$ and $r=2\rho+5$, so that every graph in $\calG$ excludes induced $(\leq 2\rho+5)$-subdivisions of $F$.
Let $t=\max(\size{V(F)},\size{E(F)})$.
Since every proper $(\leq\rho+2)$-subdivision of $K_{t,t}$ contains an induced $(\leq 2\rho+5)$-subdivision of $F$, the graphs in $\calG$ exclude induced proper $(\leq\rho+2)$-subdivisions of $K_{t,t}$.
Consequently, by \cref{lem:control-reduction}, the class $\calG$ is $2$-controlled.
\end{proof}

\section{Algorithmic consequences}
\label{sec:effective}

We proceed with stating algorithmic extensions of \cref{thm:string-b-bounded,thm:region-b-bounded,thm:region-asdim,thm:asdim-b-bounded,thm:control-b-bounded}.
In all these results, $n$ denotes the total size of the input.

An \emph{intersection model} of a string graph $G$ is a region intersection model of $G$ over a planar graph, which describes combinatorially the family of curves in the plane that give rise to the graph $G$.
Thus, if $G$ is assumed to be given along with an intersection model, the total size $n$ of the input is bounded by a polynomial function of $\size{V(G)}+\size{V(H)}$, where $H$ is the underlying planar graph.

Here is an algorithmic extension of \cref{thm:string-b-bounded}.

\begin{theorem}
\label{thm:string-b-bounded-alg}
There are functions\/ $f,C\colon\setN\times\setN\to\setN$ such that the following holds.
There is an algorithm that, given a string graph\/ $G$ and an intersection model of\/ $G$ in the input, in\/ $n^{C(\omega(G),\beta(G))}$ time computes a proper coloring of\/ $G$ with at most\/ $f(\omega(G),\beta(G))$ colors.
\end{theorem}

The requirement that the intersection model of $G$ is provided in the input is strong, because there exist string graphs $G$ for which all such models are of size exponential in the size of $G$~\cite{KM91}.
However, we have been unable to avoid it, and this seems to be a general phenomenon in the area \cite{AHW19,BR25,CPW24,MP22}.
Intersection models of segment graphs have size linear in the size of the graph.

Here is an analogous extension of \cref{thm:region-b-bounded}.

\begin{theorem}
\label{thm:region-b-bounded-alg}
There are functions\/ $f,C\colon\setN\times\setN\times\setN\to\setN$ such that the following holds.
There is an algorithm which, given a number\/ $t\geq 2$, a graph\/ $H$ excluding\/ $K_t$ as a minor, a region intersection graph\/ $G$ over\/ $H$, and a region intersection model of\/ $G$ over\/ $H$ in the input, in\/ $n^{C(t,\omega(G),\beta(G))}$ time, computes a proper coloring of\/ $G$ with at most\/ $f(t,\omega(G),\beta(G))$ colors.
\end{theorem}

The \emph{radius} of a set of vertices $S$ in a graph $G$ is $\min_{x\in V(G)}\max_{y\in S}d_G(x,y)$.
The radius of a set is at most its diameter and at least half of the diameter.
\Cref{thm:region-asdim} can be made into a linear-time \textsf{FPT} algorithm, which follows directly from the proof in \cref{sec:region}.

\begin{theorem}
\label{thm:region-asdim-alg}
There is a function\/ $C\colon\setN\to\setN$ such that the following holds.
There is an algorithm which, given a number\/ $t\geq 2$, a graph\/ $H$ excluding\/ $K_t$ as a minor, a region intersection graph\/ $G$ over\/ $H$, a region intersection model of\/ $G$ over\/ $H$, and a number\/ $\delta\geq 0$ in the input, in\/ $C(t)n$ time, computes a\/ $(2^{t-2},\delta)$-disjoint partition of\/ $V(G)$ into sets of radius at most\/ $2(t-1)\delta+1$ in\/ $G$.
\end{theorem}

For a graph $G$, a number $d\geq 0$, and a function $f\colon\setN\to\setN$, a \emph{$d$-dimensional\/ $f$-control oracle} for $G$ is an oracle which, given a set $S\subseteq V(G)$ and a number $\delta\geq 0$ in the input, provides a $(d+1,\delta)$-disjoint partition of $S$ into sets of diameter at most $f(\delta)$ in $G[S]$.
For instance, \cref{thm:region-asdim-alg} can be used to turn a number $t\geq 2$, a graph $H$ excluding $K_t$ as a minor, a region intersection graph $G$ over $H$, and a region intersection model $\mu$ of $G$ over $H$ into a $(2^{t-2}-1)$-dimensional $(\delta\mapsto 4(t-1)\delta+2)$-control oracle for $G$ that encloses $t$, $H$, $G$, and $\mu$ in some kind of internal storage.
Such an oracle then works as follows: given a set $S\subseteq V(G)$ and a number $\delta\geq 0$ in the input, it runs the algorithm claimed in \cref{thm:region-asdim-alg} on $t$, $H$, $G[S]$, $\mu|_S$, and $\delta$ to compute a $(2^{t-2},\delta)$-disjoint partition of $S$ into sets of radius at most $2(t-1)\delta+1$ (so diameter at most $4(t-1)\delta+2$) in $G[S]$.

Just like \cref{thm:region-asdim,thm:asdim-b-bounded} imply \cref{thm:region-b-bounded}, \cref{thm:region-b-bounded-alg} is a consequence of \cref{thm:region-asdim-alg} and the following extension of \cref{thm:asdim-b-bounded}.
(Note that when stating the running time of an algorithm that uses an oracle, we omit the time needed for the oracle to produce the answer to a query, as if the oracle was doing that instantly.)

\begin{theorem}
\label{thm:asdim-b-bounded-alg}
There are functions\/ $f,C\colon\setN\times\setN^\setN\times\setN\times\setN\to\setN$ such that the following holds.
There is an algorithm which, for every function\/ $f'\colon\setN\to\setN$, given a number\/ $d\geq 0$ and a graph\/ $G$ in the input, and given access to a\/ $d$-dimensional\/ $f'$-control oracle for\/ $G$, in\/ $n^{C(d,f',\omega(G),\beta(G))}$ time, computes a proper coloring of\/ $G$ with at most\/ $f(d,f',\omega(G),\beta(G))$ colors.
\end{theorem}

Algorithmic versions of results on $\rho$-controlled graph classes are even more convoluted, as they require two kinds of oracles.
For $\rho\geq 2$, a \emph{$\rho$-bounded coloring oracle} for a graph $G$ is an oracle which, given a set $S\subseteq V(G)$ and a set $S'\subseteq S$ of radius at most $\rho$ in $G[S]$ in the input, provides a proper coloring of $G[S']$.
For $\rho\geq 2$ and a function $f\colon\setN\to\setN$, an \emph{$(f,\rho)$-control oracle} for $G$ is an oracle which, given a set $S\subseteq V(G)$ in the input, and given access to a $\rho$-bounded coloring oracle for $G$, provides a proper coloring of $G[S]$ with at most $f\bigl(\max_{v\in S}\chi(\smash[b]{N_{G[S]}^\rho}[v])\bigr)$ colors.
Note that an algorithm that uses an $(f,\rho)$-control oracle for $G$ need to provide a $\rho$-bounded coloring oracle.

A simple example of a $(c\mapsto mc,\rho)$-control oracle for $G$ and a constant $m$ is one that computes an $(m,1)$-disjoint partition $\calU=\bigcup_{i=1}^m\calU_i$ of $S$ into sets of radius at most $\rho$ in $G[S]$, where $\calU_i$ is $1$-disjoint for each $i\in[m]$; the oracle then calls the $\rho$-bounded coloring oracle to obtain proper colorings of these sets, for each $i\in[m]$ using a separate set of colors on the sets in $\calU_i$.
This way, a $d$-dimensional $f$-control oracle can be turned into a $(c\mapsto(d+1)c,f(1))$-control oracle for $G$.
Thus, just like \cref{thm:control-b-bounded} implies \cref{thm:asdim-b-bounded}, \cref{thm:asdim-b-bounded-alg} is a consequence of the following extension of \cref{thm:control-b-bounded}.

\begin{theorem}
\label{thm:control-b-bounded-alg}
There are functions\/ $f,C\colon\setN\times\setN^\setN\times\setN\times\setN\times\setN\to\setN$ such that the following holds.
There is an algorithm which, for every function\/ $f'\colon\setN\to\setN$, given numbers\/ $t\geq 1$ and\/ $\rho\geq 2$ and a graph\/ $G$ excluding induced\/ $(\leq\rho+2)$-subdivisions of\/ $K_{t,t}$ in the input, and given access to an\/ $(f',\rho)$-control oracle for\/ $G$, in\/ $n^{C(\rho,f',t,\omega(G),\beta(G))}$ time, computes a proper coloring of\/ $G$ with at most\/ $f(\rho,f',t,\omega(G),\beta(G))$ colors.
\end{theorem}

We have provided the proof of \cref{thm:region-asdim-alg} in \cref{sec:region}.
\cref{thm:string-b-bounded-alg,thm:region-b-bounded-alg,thm:asdim-b-bounded-alg,thm:control-b-bounded-alg} can be proved by inspecting the proofs of the corresponding existential statements---\cref{thm:string-b-bounded,thm:region-b-bounded,thm:asdim-b-bounded,thm:control-b-bounded}, and by verifying that they provide polynomial-time algorithms.
We have structured our proofs in a way that makes this task pretty straightforward.
However, the same work needs to be done for the proofs of \cref{lem:control-reduction,lem:independent-multicover} and of \cref{thm:string-2-controlled} in~\cite{CSS21}, and for the proof of the result in~\cite{CSS17} that contains the core of the proof of \cref{lem:independent-multicover}, as they are essential in our proofs of \cref{thm:string-b-bounded-alg,thm:region-b-bounded-alg,thm:asdim-b-bounded-alg,thm:control-b-bounded-alg}.
We have verified that suitable algorithmic versions of \cref{lem:control-reduction,lem:independent-multicover} and of \cref{thm:string-2-controlled} indeed follow from the proofs in~\cite{CSS17,CSS21}.

\section{Concluding remarks}
\label{sec:conclusion}

Burling graphs have been used to show that a number of graph classes are not $\chi$-bounded \cite{CELO16,CKP25,Dav23,FJM+18,PKK+13,PKK+14,PT24a,PT24b}.
In this paper, we introduced new techniques for instead showing Burling-control.
Morally, Burling graphs can now be seen as a weak counterexample to a graph class being $\chi$-bounded, since it leaves Burling-control open.
Thus, Burling-control should now be considered for the graph classes for which Burling graphs were used as a counterexample to $\chi$-boundedness.
Beyond the classes for which we have already proved Burling-control in this paper, we discuss a few more.

The original application of Burling graphs by Burling~\cite{Bur65} in 1965 was to show that intersection graphs of axis-parallel boxes in $\setR^3$ are not $\chi$-bounded.
However there are now other such constructions for box intersection graphs \cite{Dav21,MM11,RA08}, showing that they are not Burling-controlled.

A \emph{wheel} is a graph consisting of an induced cycle of length at least $4$ and a single additional vertex adjacent to at least $3$ vertices on the cycle.
Trotignon~\cite{Tro13} conjectured that graphs without an induced wheel are $\chi$-bounded.
However, this and a more general conjecture of Scott and Seymour~\cite{SS20c} were disproved using Burling graphs~\cite{Dav23,PT23}.
We conjecture that Burling graphs are the only such counterexample.

\begin{conjecture}
The class of graphs with no induced wheel is Burling-controlled.
\end{conjecture}

As discussed, Burling graphs were used to disprove Scott's conjecture that the class of graphs forbidding induced subdivisions of a fixed graph $F$ is $\chi$-bounded.
Since then, Burling graphs have been used to study more closely for which graphs $F$ Scott's conjecture is true \cite{CELO16,PT24a,PT24b}.
The most general positive result (generalizing trees and long cycles) is that Scott's conjecture is true when $F$ is a banana tree, that is, a graph obtained from a tree by replacing every edge $uv$ by some number of internally vertex disjoint paths connecting $u$ and $v$~\cite{SS20a}.
A corollary to \cref{conj:subdivision-b-bounded} would be that Scott's conjecture is true for $F$ if and only if a subdivision of $F$ is a Burling graph.

As a corollary of \cref{thm:b-bounded} and some $2$-control results of Scott and Seymour~\cite{SS20a}, we prove special cases of \cref{conj:subdivision-b-bounded} for some new graphs $F$, and also some new special cases of Scott's original, false conjecture.
We say that a graph is a \emph{bloated cycle} if it is obtainable from a cycle by replacing all but one edge by a collection of internally disjoint paths between its endpoints.
A graph is a \emph{bloated triangle} if it is obtainable from a triangle by replacing two of its edges by a collection of internally disjoint paths between its endpoints, and its third edge by a collection of at most two internally disjoint paths between its endpoints.
Lastly, we say that a graph is a \emph{bloated bridge} if it is obtainable from a graph consisting of a path $v_1,\ldots,v_k$ with $k\geq 4$ by adding two vertices $a,b$ with edges $ab,av_1,av_2,bv_{k-1},bv_k$ by replacing each edge $v_iv_{i+1}$ by a collection of internally disjoint paths between its endpoints.
Scott and Seymour~\cite{SS20a} proved that excluding induced subdivisions of any one of these graphs gives a graph class that is $\rho$-controlled for some $\rho\geq 2$ and therefore, by \cref{lem:control-reduction}, $2$-controlled.

\begin{theorem}[Scott, Seymour {\cite[7.2, 8.2, 9.1]{SS20a}}]
\label{thm:bloated-control}
If\/ $F$ is a bloated cycle, triangle, or bridge, then the class of graphs excluding induced subdivisions of\/ $F$ is\/ $2$-controlled.
\end{theorem}

Combining \cref{thm:bloated-control} with \cref{thm:b-bounded}, we obtain the following.

\begin{theorem}
\label{thm:bloated-b-bounded}
If\/ $F$ is a bloated cycle, triangle, or bridge, then the class of graphs excluding induced subdivisions of\/ $F$ is Burling-controlled.
\end{theorem}

\Cref{thm:bloated-b-bounded} also yields some new $\chi$-bounded graph classes.
We give one such example.
Let $C_5^<$ be the graph obtained from the $5$-cycle by adding two chords that share a vertex.
Since every subdivision of $C_5^<$ has a subdivision which is a Burling graph \cite[Theorem~4.5]{PT24b}, we obtain the following as a consequence of \cref{thm:bloated-b-bounded}.

\begin{theorem}
If\/ $F$ is a subdivision of\/ $C_5^<$, then the class of graphs excluding induced subdivisions of\/ $F$ is\/ $\chi$-bounded.
\end{theorem}

Since the class of sphere intersection graphs in $\setR^d$ has finite asymptotic dimension for every $d\geq 2$~\cite{DGHM25}, it follows by \cref{thm:asdim-b-bounded} that these classes are Burling-controlled.
(Note that they do contain Burling graphs~\cite{PKK+13}.)
This can be seen as a high-dimensional generalization of the fact that circle graphs are $\chi$-bounded~\cite{Gya85}.

\begin{theorem}
The class of sphere intersection graphs in\/ $\setR^d$ is Burling-controlled for every\/ $d\geq 2$.
\end{theorem}

It is easy to see that for every proper minor-closed graph class $\calH$ and every graph $F\notin\calH$, region intersection graphs over $\calH$ exclude proper subdivisions of $F$ as induced minors.
The following is a variant of a question of Bonamy et~al.\ \cite[Question~5]{BBE+23} and of a conjecture of Georgakopoulos and Papasoglu \cite[Conjecture~1.3]{GP25}; it strictly generalizes \cref{thm:region-asdim} in view of the result of~\cite{BH25}.

\begin{conjecture}
\label{conj:induced-minors-asdim}
For every graph\/ $F$, the class of graphs excluding\/ $F$ as an induced minor has finite asymptotic dimension.
\end{conjecture}

In view of \cref{thm:asdim-b-bounded}, the conjecture above would imply that the class of graphs excluding any fixed graph $F$ as an induced minor is Burling-controlled.
Albrechtsen, Jacobs, Knappe, and Wollan~\cite{AJKW25} proved the special case of \cref{conj:induced-minors-asdim} for graphs $F$ that are subdivisions of $K_4$ (showing that the asymptotic dimension is at most $1$).
Therefore, we obtain the following.

\begin{theorem}
If\/ $F$ is a subdivision of\/ $K_4$, then the class of graphs excluding\/ $F$ as an induced minor is Burling-controlled.
\end{theorem}

In \cref{thm:string-b-bounded-alg}, we claim that an $f(\omega(G),\beta(G))$-coloring of a string graph $G$ given in the input along with its intersection model can be computed in $n^{C(\omega(G),\beta(G))}$ time.
We conjecture that this can be achieved without requiring an intersection model in the input.

\begin{conjecture}
\label{conj:string-b-bounded-alg}
There are functions\/ $f,C\colon\setN\times\setN\to\setN$ and an algorithm that takes a graph\/ $G$ of size\/ $n$ in the input and in\/ $n^{C(\omega(G),\beta(G))}$ time either computes an\/ $f(\omega(G),\beta(G))$-coloring of\/ $G$ or correctly reports that\/ $G$ is not a string graph.
\end{conjecture}

Algorithms with specification like in \cref{conj:string-b-bounded-alg} are called \emph{robust}; see~\cite{RS03}.
We remark that a positive and algorithmizable solution to \cref{conj:induced-minors-asdim} is likely to imply a positive solution to \cref{conj:string-b-bounded-alg}.

Although it is \NP-hard in general~\cite{Kar72}, in many classes of graphs, the maximum size of a clique can be computed in polynomial time \cite{BJK+18,CZ17,CCJ90,Gav73,GLS84,Spi03}.
For such $\chi$-bounded graph classes, this allows for a polynomial-time approximation algorithm for the chromatic number.
If one could determine (or at least approximate) both $\omega(G)$ and $\beta(G)$ in a Burling-controlled graph class, then a polynomial-time approximation algorithm for the chromatic number in the class would also be obtained.
So far, we do not know the complexity of determining or approximating $\beta(G)$, even when $G$ is an arbitrary graph, or a Burling graph.
A polynomial-time approximation algorithm for $\beta(G)$ when $G$ is a Burling graph would yield a polynomial-time approximation algorithm for the chromatic number of Burling graphs, which remains open.

In \cref{thm:burling-minimal}, we proved that Burling graphs form a minimal hereditary class of graphs with unbounded chromatic number, just like the class of complete graphs.
No such minimal class other than the two is known.
The notion of Burling-control highlights the importance of such minimal classes.
Other classical constructions of triangle-free graphs with large chromatic number appear not to be such minimal classes.
It would be exciting to have any further examples of a minimal hereditary class of graphs with unbounded chromatic number, as they are bound to shed further light on the fundamentals of what causes a graph to have large chromatic number.

\section*{Acknowledgements}

Significant part of this work was done during the Structural Graph Theory Workshop STWOR2, which took place in Chęciny, Poland, in July 2024.
The workshop was supported by the Excellence Initiative -- Research University (IDUB) funds of the University of Warsaw, as well as the project BOBR that is funded from the European Research Council (ERC) under the European Union's Horizon 2020 research and innovation programme with grant agreement No.\ 948057.
We are grateful to the organizers and other participants for a nice and productive atmosphere.

\bibliographystyle{plain}
\bibliography{burling}

\end{document}